\documentclass[11pt,reqno]{amsart}
\usepackage{xifthen,setspace}

\usepackage{graphicx,subfigure,amsmath,amssymb,amsfonts,bm,cite,epsfig,epsf,url,alg,dsfont}
\usepackage{graphicx,subfigure,amsmath,amssymb,amsfonts,bm,epsfig,epsf,url,dsfont}
\usepackage{mathrsfs}
\usepackage{times}
\usepackage{bbm}      
\usepackage{booktabs}
\usepackage{cases}
\usepackage{fullpage}
\usepackage[small,bf]{caption}
\usepackage[top=1in,bottom=1in,left=1in,right=1in]{geometry}
\usepackage{stmaryrd}

\numberwithin{equation}{section}

\def\A{\mathcal{ A}_r}
\def\P{\mathbb{P}}

\def\N{\mathbb{N}}
\def\Z{\mathbb{Z}}

\newcommand{\e}{{\varepsilon}}

\newcommand{\E}{\mathbb{E}}
\renewcommand{\P}{\mathbb{P}}

\newcommand{\cP}{\mathcal{P}}
\newcommand{\cD}{\mathcal{D}}
\newcommand{\cA}{\mathcal{A}}
\newcommand{\cR}{\mathcal{R}}
\newcommand{\cB}{\mathcal{B}}

\newcommand{\sD}{\mathscr{D}}

\newcommand{\sF}{\mathscr{F}}
\newcommand{\sG}{\mathscr{G}}

\newcommand{\sI}{\mathscr{I}}

\newcommand{\sP}{\mathscr{P}}

\newtheorem{theorem}{Theorem}
\newtheorem{corollary}[theorem]{Corollary}

\newtheorem{rem}[theorem]{Remark}
\newtheorem{lemma}[theorem]{Lemma}
\newtheorem{definition}[theorem]{Definition}
 
\begin{document}

\title{Non-fixation for Conservative Stochastic Dynamics on the Line}
\date{\today}
\author{Riddhipratim Basu}
\address{Riddhipratim Basu, International Centre for Theoretical Sciences, Tata Institute of Fundamental Research, Bangalore, INDIA}
\email{rbasu@icts.res.in}
\author{Shirshendu Ganguly}
\address{Shirshendu Ganguly, Departments of Statistics and  Mathematics, UC Berkeley, CA, USA}
\email{sganguly@berkeley.edu}
\author{Christopher Hoffman}
\address{Christopher Hoffman, Department of Mathematics, University of Washington,
Seattle, USA}
\email{hoffman@math.washington.edu}

\maketitle
\begin{abstract} We consider Activated Random Walk (ARW), a model which generalizes the Stochastic Sandpile, one of the canonical examples of self organized criticality. Informally ARW is a particle system on $\Z$ with mass conservation. One starts with a mass density $\mu>0$ of initially active particles, each of which performs a symmetric random walk at rate one and falls asleep at rate $\lambda>0.$ Sleepy particles become active on coming in contact with other active particles. We investigate the question of fixation/non-fixation of the process and show for small enough $\lambda$ the critical mass density for fixation is strictly less than one. Moreover, the critical density goes to zero as $\lambda$ tends to zero. This settles a long standing open question. 
\end{abstract}

\section{The model description and main result}\label{des1}
Self organized criticality (SOC) is a universal property describing systems that have critical points as an attractor. In most of the examples of SOC,  simple local moves give rise to complex global properties. These systems are driven under their natural evolution to  the boundary between  stable and unstable states without  any fine-tuning of parameters. A canonical and widely studied example of such a model is the Abelian Sandpile model proposed by Bak, Tang and Weisenfeld \cite{BTW87}. In this model, in finite volume, particles are added at random, which dissipate across the boundary; the corresponding closed system in the infinite volume setting  is a so-called \emph{fixed energy sandpile} where the total number of particles is conserved. 

In a sequence of influential papers Dickman, Vespignani and their co-authors \cite{DVZ98, VDMZ98, VDMZ00, SOC2} developed a theory of SOC, which has since become widely accepted in statistical physics literature. Their theory predicts a specific relationship between driven dissipative systems and the corresponding systems where total number of particles is conserved; and in particular an \emph{absorbing state phase transition}, exhibited by the latter system as the particle density is varied. In the finite state, even though there is no tuning parameter, they argued that the loss of particles through the sink(s) is balanced with the dynamical addition of new particles, driving the system to the edge of instability.  However subsequently, there has been some controversy in the mathematics and physics community surrounding their predictions. We elaborate more on this in \S~\ref{s:background}.

In the last fifteen years, there has also been significant progress in studying these systems via a combinatorial approach, in particular exploiting the so-called Abelian property. Roughly, the Abelian property in a distributed network stipulates that the system produces the same output regardless of the order in which a sequence of local moves are performed. A systematic mathematical treatment of such systems can be found in \cite{chip}. Even though this might seem a severe restriction, systems with Abelian property are known to produce intricate large-scale patterns starting with relatively simple local rules. In addition to the Abelian sandpile model and a number of cellular automata introduced by Dhar \cite{Dhar} exhibiting SOC, other systems in this class include Internal Diffusion Limited Aggregation (IDLA) and rotor walks (see e.g., Levine, Peres \cite{LP09} and the references therein) which have been subject of extensive studies in recent years, out of which a rich structure has emerged. See Bond, Levine\cite{bond} for an excellent survey on these so-called Abelian networks.

Several variations of the Abelian Sandpile model have been studied in the past with a view to better understand the critical behaviour. Some recent studies (e.g., Bonachela, Mu\~noz \cite{BM08}) have focussed their attention to sandpile models with stochastic update rules. Although the paradigm of Dickman, Vespignani et al. for SOC is widely believed to contain both deterministic and stochastic sandpile models \cite{VD05, dCVdSD09}, adding randomness can lead to a qualitatively different critical behaviour. The best known example in this random setting is the \textbf{Stochastic Sandpile model} (SSM). On the graph $\Z$, this model starts with particles at every site distributed independently from a fixed distribution  with mean $\mu$. Any unstable site (a site with at least two particles) topples by emitting two particles which take independently one step of the simple random walk.  The most important phenomenon in the study of \rm{SSM} is the phase transition related to the  long term fixation/non-fixation of the system. Empirically when the density is sufficiently small, one observes  activity in any finite window about the origin only finitely often, while when the density is sufficiently large, the system does not fixate.

Getting nontrivial bounds on the critical behaviour has turned out to be extremely challenging. To make things mathematically tractable  a closely related family of continuous time interacting particle systems known as \textbf{Activated Random Walk} was introduced. 
This system consists of particles which are in one of two states, active or sleepy.
Activated random walk starts with particles distributed independently on the sites of a graph  from a distribution with mean $\mu.$  Initially all of the particles are in the active state.
 As the system evolves each active particle performs a continuous time nearest neighbour symmetric random walk  at rate $1$ until it falls asleep, which happens with rate $\lambda$. A particle in the sleepy state does not move but it becomes active whenever an active particle occupies its site.
We denote this system on $\Z$ by $\rm{ARW} (\mu,\lambda)$.

 Note that the case when $\lambda$ is infinity is a variant of SSM. For a fixed sleep rate $\lambda$, as the particle density $\mu$ increases, it is expected that the system shows a transition from almost sure local fixation to staying active forever almost surely. One of the first mathematically rigorous results about ARW was established in \cite{RS12} where it is shown that for every $\lambda>0$ there is a critical particle density $\mu_c:=\mu_{c}(\lambda)\in [\frac{\lambda}{\lambda+1},1]$ such that ${\rm ARW}(\mu,\lambda)$  locally fixates almost surely when $\mu<\mu_c(\lambda)$ and stays active almost surely when $\mu>\mu_c(\lambda)$. Quoting from \cite[Section 7]{RS12},
\begin{quote}
\emph{``A proof that $\mu_{c} < 1$  for the SSM and ARW remains as an open problem in any dimension. For the ARW, it should hold for all $\lambda$, and moreover $\mu_{c}(\lambda)\to 0$ as
$\lambda \to 0$. Yet, even a proof that $\mu_{c}(\lambda) < 1$ for some $\lambda > 0$ is missing."}
\end{quote}
 
These problems have been re-iterated in \cite{DRS10}, \cite{CRS14} and \cite{T14}. In this article we only look at $\rm{ARW}$ and  for easy reference purpose we now record below the aforementioned conjectures in our setting.  \\

\textbf{Conjecture 1.} For $\rm{ARW}$ on $\Z$, for all sleep rate $\lambda> 0$, the critical mass density $\mu_c{(\lambda)}< 1.$ \\

\textbf{Conjecture 2.} In the same set up as \textbf{Conjecture 1}, $\mu_c{(\lambda)}\to 0$ as $\lambda\to 0$.\\

The main result of this paper  immediately provides a positive resolution of \textbf{Conjecture 1} for small $\lambda$. 

\begin{theorem}
\label{main}
Given $\mu>0$ there exists $\lambda_{\mu}>0$ such that ${\rm ARW}(\mu,\lambda)$ on $\Z$  stays active almost surely for all $\lambda < \lambda_{\mu}$.
\end{theorem}

It was also established in \cite{RS12} that $\mu_{c}(\lambda)$ in a non-decreasing function of $\lambda$. As a consequence we obtain the following corollary resolving \textbf{Conjecture 2}:
\begin{corollary}
\label{cor:mulambda}
For ${\rm ARW}(\mu,\lambda)$ on $\Z$, the critical particle density $\mu_c{(\lambda)}\to 0$ as $\lambda \to 0$.
\end{corollary}

\begin{rem}
\label{r:more}
For any $\lambda >0$, it is easy to heuristically explain why $\mu_c({\lambda})$ should be at most one. If $\mu >1$, then on average ``there are more particles than sites", and since at most one particle can fall asleep at a site, the system should not fixate.  In different settings this argument is formalized in \cite{GA10,RS12,Shellef10}, see \S~\ref{s:background} for details. However, establishing $\mu_c({\lambda})<1$ requires understanding how activity is being sustained forever due to particle interaction, i.e.\ sleepy particles being woken up by active particles over and over. This makes the analysis substantially more difficult.     
\end{rem}

\subsection{Background}
\label{s:background}
In the recent years, studies in non-equilibrium statistical mechanics have offered up a number of mathematically challenging interacting particle processes which exhibit phase transitions far away from equilibrium. A particular class of conservative models that has drawn significant attention is one where even though the mass is conserved, particles can exist in one of the two states: active and inactive. Typically an inactive particle becoming active requires interaction with one or more active ones. Paradigm examples in this class includes conserved lattice gases \cite{Vespignani00} and sandpile models with stochastic update rules \cite{Dickman2001, Manna90, Manna91}. In infinite volume, these models are believed to exhibit \emph{absorbing-state phase transition} from an active phase as some model parameter (typically particle density) is varied. In finite volume, when run with a carefully controlled driven-dissipative mechanism, these systems are believed to exhibit the phenomenon of \emph{self-organized criticality} \cite{SOC1, SOC2}, where the system is attracted to a critical state, even though it is not explicitly tuned to this critical value.   

We digress for a moment here to remark that some recent results have raised questions about the exact relationship between criticality in the  absorbing state phase transition (fixed energy model) and driven-dissipative mechanism in this general class of models. In the context of the Abelian Sandpile model, even though in dimension one, the critical densities according to the two definitions mentioned above  both equal one, in higher dimensions, for e.g. on $\Z^2,$ Fey, Levine and Wilson \cite{Fey} presented numerical evidence and rigorous results in other related settings, suggesting that the critical densities  of the fixed energy model and the driven dissipative system perhaps differ. See \cite{Fey} for precise results and  an excellent discussion on the above topic.  A recent work of Hough, Jerison and  Levine \cite{houghsandpile} prove among other things, an improved upper bound for the critical density for the fixed energy model on $\Z^2$.  See also Levine \cite{levinethreshold} for an investigation of the relationship between the two notions of criticality.    

Coming back to fixed energy sandpile models with stochastic update rules, the transitions in these models are believed to belong to a universality class, referred to as the \emph{Manna class}. Whether the Manna class exists as an autonomous universality class separate from the universality class of directed percolation (DP) seems not to have a broadly accepted answer among physicists at this point \cite{Basu12PRL, Lee13PRL}, it appears beyond the state-of-the-art techniques to obtain mathematically rigorous results on the critical or near-critical behaviour of these models (see \cite{CRS14} for some progress). Even the more basic questions about existence of phase transitions seem challenging and has only been settled in a few particular cases. One of the main challenges in studying these systems is the intricate long-range interaction caused by the conservation of particles, which makes it harder to apply some of the standard techniques in rigorous statistical mechanics.  

Activated Random Walk (ARW) is one of the well-known examples in the above general class of models. As mentioned at the beginning, a major motivation to study ARW comes from the following Stochastic Sandpile Model (SSM), a variant of Manna's model \cite{Manna90, Manna91}. In the SSM on the line, started with an initial particle configuration of product measure with density $\mu$, a site with an isolated particle instantaneously becomes inactive, whereas at any site containing $d\geq 2$ particles remain active, and at rate $1$, emits two particles using independent symmetric random walk steps, leaving $(d-2)$ particles at the site (observe the contrast to the deterministic Abelian Sandpile model, where two particles are emitted in two neighbouring directions). Compared to the deterministic sandpile model (see e.g.\ \cite{Dhar}), much less is rigorously known about SSM. It was only recently proved in \cite{RS12} that there exists $\mu_c\in [\frac{1}{4},1]$ such that the system fixates for $\mu<\mu_c$  and remains active for $\mu> \mu_c$. Numerical simulations suggest that $\mu_c \approx 0.9489$,  it remains a major mathematical challenge to prove $\mu_c<1$. It is reasonable to expect that ARW is a sufficiently good (and possibly more mathematically tractable) approximation to SSM and thus captures some of its crucial aspects. In particular as $\lambda\to \infty$, ARW corresponds exactly to the model studied in \cite{Jain}.

ARW can also be viewed as a special case of driven-diffusive epidemic processes. In this process a healthy particle does a simple random walk at rate $D_A\geq 0$ and each infected particle does a simple random walk with rate $D_{B}>0$, infecting all particles that it steps on, and recuperating at rate $\lambda$. ARW corresponds to the special case $D_A=0$. This model, in the case $D_A=D_B$ was introduced by Spitzer in late 1970s, but was rigorously studied in detail much later in \cite{KS1,KS2,KS3,KS4}. Although numerical studies have predicted different regimes of critical behaviour for $D_{A}<D_{B}$ and $D_{A}>D_{B}$, not much is rigorously known for the case $D_{A}\neq D_{B}$ (see \cite{CRS13} for the study of a related model), in particular, it is not understood whether the behaviour for $D_{A}=0$ (i.e.\ ARW) should be the same as the behaviour for $0<D_{A}<D_{B}$. 

For more background and a fuller history of ARW and related processes, the interested reader is referred to \cite{DRS10,RS12} and the references therein; see \cite{MD99} for a detailed account of non-equilibrium phase transitions in lattice models.

There have been a flurry of rigorous results on ARW in the last few years, mostly in the wake of the breakthrough paper \cite{RS12}, where, as mentioned above, the existence of an absorbing-state phase transition was established on $\Z$ for both ARW and SSM. This paper crucially uses a construction of the ARW process using the Diaconis-Fulton representation and the Abelian property (see the next section). An upper bound of one on the critical density for ARW was also established in \cite{RS12} for any $\lambda>0$. The same upper bound was also established in \cite{GA10} for ARW on unimodular graphs using mass transport principle and in \cite{Shellef10} for general bounded degree graphs using a comparison with certain internal aggregation models. The results of \cite{Shellef10} also imply that critical density is positive for ARW for any $\Z^d$ for $\lambda=\infty$. This result has recently been established for any $\lambda>0$ and any $d$ in \cite{ST14} using a multi-scale argument. As far as the critical value is concerned the only result is by \cite{CRS14} for the case $\lambda=\infty$, where it is shown that the system on $\Z^d$ fixates for all $\mu<1$, and thus in conjunction with the results of \cite{GA10,Shellef10} this establishes $\mu_c=1$.  The behaviour of ARW at criticality is also largely open, essentially the only result is for $\lambda=\infty$, in which case the system does not fixate at criticality \cite{CRS14}. For more detailed predictions on the behaviour of ARW, see \cite{DRS10}.

So far we have concerned ourselves with the model where the active particles perform a simple symmetric random walk, and that scenario will be the focus of the present work. It is also interesting to consider a situation where the random walk steps are biased, and in certain cases, such models are better understood. The study of the biased ARW started with an unpublished argument of Hoffman and Sidoravicius (see \cite[Theorem 1]{CRS14}) that considers the case of totally asymmetric walks and establishes that the critical density is $\mu_c=\frac{\lambda}{\lambda +1}$ and further that the system does not fixate at criticality. More recently ARW on $\Z^d$ with asymmetric (but not necessarily totally asymmetric) jump distribution has been studied in \cite{T14,RT15}. It is shown in \cite{T14} that in $d=1$, when the jump distribution is biased, Conjecture 2 holds and Conjecture 1 holds for small sleep rate. In a very recent paper \cite{RT15} the same has been established for $d\geq 2$ sharpening a previous result of \cite{T14} which stated that the system does not fixate for small $\lambda$ and $\mu$ sufficiently close to one. We refer the reader to the lecture notes of Leonardo Rolla \cite{R15} for more details on the previously known results on ARW.

Despite this impressive progress in the study of biased ARW, Conjecture 1 and Conjecture 2 have so far remained open for symmetric ARW (we shall always refer to the symmetric case by ARW unless otherwise mentioned), and it appears that the methods employed in these works cannot be adapted to understand Conjecture 1 and Conjecture 2 for symmetric ARW on $\Z$. This is the main contribution of the present paper.\\

The proof considers truncated versions  of $\rm{ARW}$ on a finite universe. This technique has been used in several other arguments regarding this model in the literature. The new ingredients include a novel use of  the Abelian property (see \S~\ref{fvap}) which allows us  to study a slightly different labeled variant of $\rm{ARW}$  which is stochastically dominated by the actual $\rm{ARW}$. In other words we construct a coupling of the two systems where non-fixation of the former system implies non-fixation of the latter. More details appear in \S~\ref{sop}. 

\section{Formal definitions and setup}\label{fsetup}
Recall that ARW is the following continuous time interacting particle system on $\mathbb{Z}$. Start with  particles at every site on the line by sampling i.i.d.\ from any distribution with mean $\mu$.  Each particle can be in one of the two states $A$ (active) and $S$ (sleepy). Initially all the particles are in active state. Each active particle does a continuous time nearest  neighbour symmetric random walk on $\Z$ at rate $1.$ Sleepy particles do not move. Also each active particle undergoes the transition $A \rightarrow S$ at rate $\lambda>0$ independent of everything else. Sleepy particles undergo the transition $S+A \rightarrow 2A$ instantaneously, i.e., a sleepy particle at $x\in \Z$ becomes active instantaneously when an active particle visits the site $x$. Also the transition $A \rightarrow S$ is observed only if at that time the active particle was the only particle at its site, i.e.\ the instantaneous transitions $2A\rightarrow A+S \rightarrow 2A$ is not observed.

We follow \cite{RS12} in formally describing the set up of ARW. 

Let $\rho$ be a formal symbol (which we will use to formally denote a sleepy particle). The state space $\Omega=(\N \cup \{0\} \cup \{\rho\})^{\Z}$ is the space of all bi-infinite sequence with elements of $\N \cup \{0\} \cup \rho$\footnote{It would be convenient later to put the following order on $\N \cup \{0\} \cup \{\rho\}$: $0 < \rho < 1< 2 <\ldots$}. 
For any time $t\ge 0,$ $\eta_t \in \Omega$ will denote the state of the system i.e.\  $\eta_t(x)$ denotes the number of particles at $x\in \Z$ at time $t$. $\eta_{t}(x)=\rho$ denotes that the only particle at $x$ at time $t$ is in state $S$ (is asleep). Following notation from \cite{RS12}, we formally let $|\rho|=1$ so that irrespective of the state of the particles, $|\eta_{t}(x)|$ denotes the number of particles at site $x$ at time $t$.
For notational convenience we define the following addition and multiplication operations on $\N \cup \{0\} \cup \rho$ to describe the $A+S \rightarrow 2A$ and $A \rightarrow S$ transitions: 
\begin{align*}
\rho+0 & =0+\rho=\rho,\\
\rho+n & =n+\rho=n+1,\\
\rho.1& =\rho,\\ 
\rho.n &=n \,\text{ for }n>1.
\end{align*}
Let for $x \in \Z$ and $\eta \in \Omega, $ $A(\eta(x))$ denote the number of active particles at site $x$ in configuration $\eta.$ With the above notation, formally the process evolves as follows:  for each site $x$,   we have the transitions $\eta \rightarrow \tau_{x,y} \eta$
at rate $A(\eta_t (x))\frac{1}{2}\mathbf{1}_{|y-x|=1},$ and $\eta\rightarrow \tau_{x,\rho}\eta$ 
 at rate  $\lambda A(\eta_t (x))$, where
\begin{align}\label{instru1}
\tau_{x,y}(\eta)(z)=\left\{
\begin{array}{cc}
\eta(z)+1 & z=y\\
\eta(z)-1 & z=x\\
\eta(z) & \text{otherwise.}
\end{array}
\right.
\end{align}
Similarly
\begin{align}\label{instru2}
\tau_{x,\rho}(\eta)(z)=\left\{
\begin{array}{cc}
\rho.\eta(x) & z=x\\
\eta(z) & \text{otherwise.}
\end{array}
\right.
\end{align}
Note that the transition when a sleepy particle is awakened by an active particle, is contained in the  relation $\rho+1=2.$

Let $\P^{\nu}$ denote the law of the process started from an initial configuration distributed according to $\nu.$ Throughout the rest of the article we will focus on $\nu$ being a product measure with identical coordinate projections which in particular implies ergodicity under translations. \\

Finally we need to show that such a process is well-defined even starting with infinitely many particles. This can be done using the general theory of interacting particle systems developed in \cite{Lig85}. It can also be done using an argument of  Andjel \cite{Andjel82} which shows that (under some mild finiteness condition on $\nu$) the process is well-defined and can be approximated in a suitable sense by its finite truncations. See \cite{RS12} for more details.

We say that the system has particle density $\mu$ if $\E_{\nu}(\eta_0(0))=\mu.$ Thus as mentioned before ${\rm ARW}(\mu,\lambda)$ will denote the activated random walk process with initial density $\mu$ and sleep rate $\lambda$ (note that we choose to suppress the dependence of the distribution $\nu$ and just keep track of the particle density).

\begin{definition}
\label{d:fixate}
ARW started from any configuration  is said to \textbf{locally fixate} if for every $x \in \Z,$ $\eta_{t}(x)$ is eventually constant. Otherwise we say that the system stays \textbf{active}.
\end{definition}

\subsection{Key ideas and outline of the proofs}\label{sop} To prove Theorem \ref{main}, we first approximate the infinite system by considering truncated $\rm{ARW}$ on large but finite boxes (see Lemma \ref{equi1}).  The goal is to then show using this approximation that,  the number of times a particle gets emitted from the origin can be made larger than any finite constant with probability bounded away from $0$. More specifically we run several rounds of the truncated process with growing intervals, with the origin as one of the end points. The intervals are chosen to be growing exponentially in size and it is argued that the total number of times particles gets emitted from the origin in  the $\ell^{th}$ round dominates a $\rm{Ber}(\frac{1}{4})$ random variable. Also by construction, we ensure that the rounds are independent. Thus the total number of particles emitted at the origin up to the $\ell^{th}$ round dominates a $\rm{Bin}(\ell,\frac{1}{4})$ random variable. Using Lemma \ref{equi1}, one then argues that since the activity at the origin in the finite volume is arbitrarily large, the origin does not fixate in the infinite volume system almost surely. Thus Theorem \ref{main} follows.  This argument is made precise in Lemma \ref{zerohit} in \S~\ref{pt2}.

 Proving the claim that the total number of times particles get emitted from the origin in  the $\ell^{th}$ round dominates a $\rm{Ber}(\frac{1}{4})$  random variable (the constant $\frac{1}{4}$ is nothing special and any small enough constant would work)  involves all the key ideas in this paper. We sketch the main steps below: we show that for a large enough interval with not too few particles it is exponentially unlikely (in the size of the interval) for the finite volume system to stabilize inside the interval without any particle touching the boundary (recall that the origin was chosen to be one of the end points of the interval). Thus with probability roughly $1/2,$ particles hit the origin (probability of hitting either of the boundary points is close to $1$, and both are equally likely to happen by the underlying symmetry of the process). One can refine this argument by analyzing certain martingales associated to the center of mass of the system but for our purpose, the above crude arguments would suffice. 

The technical core of this paper consists of proving the exponential upper bound on the probability of the event mentioned above. Roughly, we show the following (which is made precise in Lemma \ref{key}): Consider a large interval $[-r,r],$ with at least $\mu r$ particles (recall that $\mu$ is the particle density). We show that for $\lambda$ small enough it is exponentially unlikely (in $r$) that the process stabilizes within the interval $[-r,r]$ (without any particle hitting either boundary).  As mentioned before, the proof includes a some what non-standard use of the Abelian property which allows the reduction to a study of a labeled variant of the $\rm{ARW}$ making things technically convenient (see \S~\ref{ver2}). We argue by re-normalizing space: i.e., consider the lattice points $K\Z=\{\ldots,-2K,-K,0,K,2K,\ldots\},$ for $K\gg \frac{1}{\mu}$ and then show that at the end of the stabilization process, for $\lambda$ small enough, it is unlikely to have more than a particle asleep in any interval of length $K$. Since we start with enough  particles, this implies that some of them must escape through the boundary with large probability.

A statistic of fundamental  importance in our analysis is the `odometer function' : the number of times a particle was emitted from a site until stabilization, (see \eqref{odo2} for a formal definition). For other uses of the odometer in the study of Abelian systems, see \cite{LP09}. Using the odometer, the remainder of the proof  proceeds as follows: fix a sequence  of non-negative integers $\underline z=\{0=z_{-r},\ldots,z_{-2K},z_{-K},z_{0},z_{K}, z_{2K},\ldots z_r=0\}$. We bound the probability that the odometer function at the points $K\Z \cap [-r,r]$ is the sequence $\underline z$ after the truncated process on $[-r,r]$  has run until stabilization (this happens in finite time almost surely).  Since by choice $z_r=z_{-r}=0,$ on the the above event,  $-r,r$ do not emit particles. The key then is to ensure that the probability bound  as a function of $\underline z$ is small enough to allow us to take an union bound over all such sequences $\underline z$ and end up with a bound of the form $e^{-cr}$.  We provide a more detailed outline of the proof in \S~\ref{s:proofkey}

The above argument relies heavily on a  variant of the standard Diaconis-Fulton representation of Abelian systems (see \cite{DF91}).  One way to run the $\rm{ARW}$ dynamics, as done in \cite{RS12}, is to start with a probability space where at each site on $\Z$ one has `stacks' of left, right or sleep instructions used by particles when they are emitted, (see \eqref{stack}). The odometer function at a site is then the number of elements from the corresponding stack used until stabilization. For our purposes we introduce `renormalized' variants of the above where the stacks are only at the renormalized lattice points $K\Z$ and instead of being steps of length one, are lazy random walk paths stopped on hitting the nearest point in the set $K\Z.$ The laziness of the walk corresponds to the sleep instructions. See \eqref{stack2} for precise definitions.

\subsection{Organization of the paper}\label{op} In \S~\ref{fvap} we collect all the basic preliminaries about the truncated (finite universe) ARW and the Abelian property. In \S~\ref{pt2}, we show how to complete the proof of Theorem \ref{main} assuming the statement of the key Lemma \ref{key}. In \S~\ref{ver2} we develop a labeled variant of the ARW dynamics and reduce Lemma \ref{key} to Lemma \ref{key1} about the Labeled ARW dynamics. \S~\ref{s:prelim} develops a few technical preliminaries needed for the proof of Lemma \ref{key1}, which is completed in \S~\ref{s:proofkey} using a careful union bound over the possible odometer values. We finish with some open questions in \S~\ref{s:question}.

\section{Finite volume dynamics and Abelian property}\label{fvap}
For our purposes we now define a truncated version of the ${\rm ARW}(\mu,\lambda)$ dynamics. Recall that the initial configuration $\eta_0$ is distributed according to some product measure $\nu$ with particle density $\mu$.  Let $\nu_{M}$ be the restriction of $\nu$ on the finite box $[-M,M]$ and $\P^{M}:=\P^{\nu_M}$ be the corresponding law of the process. The next result shows that to prove Theorem \ref{main} it suffices to consider the measures $\P^{M}$. 

To see this consider the following modification of the ARW dynamics. Given a configuration $\eta$ on $\Z$, and $M\in \N$, let $\widehat{\eta}_M$ denote the restriction of $\eta$ on $[-M,M]$. With the initial configuration $\widehat{\eta}_M$, consider ARW process with the restriction that particles that move out of $[-M,M]$ are deleted from the system (alternatively they fall asleep immediately irrespective of everything else). Since initially there are (almost surely) finitely many particles in the system, and since any particle not falling asleep will eventually exit $[-M,M]$ it is easy to argue that this process will stabilize almost surely.

The object of interest at this point is $u_M(0)$, the number of transitions happening at the origin, (formal definitions appear in the next subsection). One now considers a coupling of all the truncated ARWs on one common probability space. We denote the underlying measure by $\mathscr{P}$ (formal definition appears in \S~\ref{s:ver1}). An important consequence of the so called Abelian property is that  under the above coupling, it follows  that the sequence  $u_{M}(0)$ is nondecreasing in $M$.
Let 
\begin{equation}\label{monlim}
u(0):=\lim_{M \to \infty}u_M(0).
\end{equation} 
  
We then have the following lemma.

\begin{lemma}\cite[Lemma 4]{RS12}
\label{equi1} 
Let $\nu$ be a translation invariant ergodic measure with finite particle density $\E_{\nu}(\eta(0))< \infty.$  Then $$\P^{\nu}(\textit{the system locally fixates})=\mathscr{P}(u(0)< \infty) \in \{0,1\}.$$ 
\end{lemma}

We refer to \cite{RS12v2}, which is an older arXiv version of \cite{RS12}, for a proof of the above Lemma. As already mentioned before the $0-1$ law is a direct consequence of ergodicity.

 In the next subsection we formally define the coupling of the truncated processes mentioned above and precisely state the monotonicity property used to define $u(0).$

\subsection{Coupling of truncated ARW's}
\label{s:ver1}
To construct the coupling of truncated ARWs we shall take resort to a Diaconis-Fulton representation \cite{DF91,Eriksson96} of this process where the process is implemented through a sequence of instructions attached to the sites. The advantage of this representation is the  Abelian property, which allows one to disregard the order in which different steps were performed in certain settings. The relevant elements of the Diaconis-Fulton representation in our context was developed in \cite[Section 3]{RS12}, and we closely follow their treatment for the rest of this and the next subsection. We start by introducing a series of notations. Recall the transitions $\tau_{x,y}$ and $\tau_{x,\rho} $ from \eqref{instru1} and \eqref{instru2}. Now consider the following array of random variables:
\begin{equation}\label{stack}
\mathscr{I}=
\begin{array}{ccccccc}
\ldots&\xi_{(-2,1)}&\xi_{(-1,1)}&\xi_{(0,1)}&\xi_{(1,1)}&\xi_{(2,1)}&\ldots\\
\ldots&\xi_{(-2,2)}&\xi_{(-1,2)}&\xi_{(0,2)}&\xi_{(1,2)}&\xi_{(2,2)}&\ldots\\
\vdots&\vdots&\vdots&\vdots&\vdots&\vdots&\vdots
\end{array}
\end{equation}
where $\xi_{(x,j)}$ are independent for any $x \in \Z$ and $j\in \N$ and moreover, 
\begin{equation}\label{law1}
\xi_{(x,j)}=\left\{ \begin{array}{cc}
\tau_{x,x-1} & \text{with probability~} \frac{1}{2(\lambda+1)}\\
\tau_{x,x+1} & \text{with probability~} \frac{1}{2(\lambda+1)}\\
\tau_{x,\rho} & \text{with probability~} \frac{\lambda}{\lambda+1}.
\end{array}
\right.
\end{equation} 
 
We would call the $\xi_{(x,j)}$'s as instructions at the site $x$ and denote the underlying product measure by $\mathscr{P}$. Using these instructions one can define a discrete time version of the ARW process in the following way:
We start by defining the notion of stability of a site $x$ with respect to a configuration $\eta \in \Omega.$

\begin{definition}
\label{stable1} 
For any $\eta \in \Omega$ we say a site $x\in \Z$ is \textbf{stable} if $\eta(x)=0$ or $\rho,$ and otherwise we call it \textbf{unstable}.
\end{definition}

Thus given a configuration $\eta \in \Omega,$ at each discrete time step $t$,  one can choose an unstable site $x$ and use the first unused element from the stack $\xi_{(x,\cdot)}$ and use it to perform the transition to a configuration $\eta'$ at time step $(t+1)$. We call such an operation ``toppling" at site $x$.
Formally we keep track of the number of topplings at every site  as a function of time $t$.
Let 
\begin{equation}\label{odo12}
h:=(h(x): x \in \mathbb{Z})
\end{equation}
where $h(x)$ denote the number of topplings at $x$,  which we will call the odometer function at $x$. $h_{t}(\cdot)$ will be used to denote the odometer function at time $t.$  For later purposes it will be convenient to keep track of both the odometer function $h_{t}$ and the configuration $\eta_t.$ To this end define the toppling operation at $x$ acting on the pair $(\eta,h)$ by $$\Phi_x(\eta,h)= (\xi_{(x,h(x)+1)}(\eta),h+\delta_x),$$ i.e.\ we use $\xi_{(x,h_{x}+1)}$, the first unused element of the stack at $x$ to topple, and $h$ increases by $1$ at $x$.
\begin{definition}
\label{d:legal}
We say that $\Phi_x$ is \textbf{legal} if $x$ is unstable i.e.\ $\eta(x)\ge 1$. For any sequence ${\bf{\alpha}}=(x_1,x_2,\ldots, x_k)$ we define the sequence of topplings at $x_1,$ followed by $x_2$ and so on through until $x_k$ by $\Phi_{\bf{\alpha}},$ i.e.\ $\Phi_{\bf{\alpha}}=\Phi_{x_k}\ldots \Phi_{x_1}$. Also for any configuration $\eta$ we will use $\Phi_{\bf{\alpha}}(\eta)$ to denote the configuration obtained by the action $\Phi_{\bf{\alpha}}$ on $\eta.$ We  say that ${\bf{\alpha}}$ is a {\bf legal sequence} if $\Phi_{x_i}$ is legal for the configuration  $\Phi_{i-1}\ldots \Phi_1(\eta)$ for all $i=1,\ldots k.$
\end{definition}

We abuse notation a little to denote by $h_{\alpha}$ the odometer function after performing the sequence of toppling given by $\alpha,$ i.e.\ for any $x \in \Z,$ 
\begin{equation}\label{odo2}
h_{\alpha}(x)=\sum_{i=1}^{k}\mathbf{1}(\alpha_i=x).
\end{equation} 
We also define the natural ordering on $\Omega$ and the space of odometer functions: for $\eta ,\tilde \eta \in \Omega,$ we say that $\eta\ge \tilde \eta,$ if $\eta(x)\ge \tilde \eta(x)$ (according to the order specified before on $\N \cup \{0\}\cup \{\rho\}$) for all $x \in \mathbb{Z}.$ Similarly for odometer functions $h,\tilde h$, we say $ h \ge \tilde h $ if $h(x)\ge \tilde h(x)$ for all $x \in \mathbb{Z}.$ We also write $(\eta,h)\ge (\tilde \eta,\tilde h)$ if $\eta \ge \tilde \eta$ and $h=\tilde h.$

The most basic property of the above process is the Abelian property which says that given two sequence of legal topplings which result in the same odometer function (see \eqref{odo2}), the final configuration is the same in both the cases i.e.\ the order in which topplings are performed does not matter. 
\begin{lemma}(Abelian property)
\label{abp} 
Given any two legal sequence of topplings $\alpha$ and $\alpha'$ such that $h_{\alpha}=h_{\alpha'}$, then $$\Phi_{\alpha}(\eta)=\Phi_{\alpha'}(\eta).$$
\end{lemma}

To see why the above is true notice that for each site $x,y\in \Z$  with $x \neq y$ any pair of legal topplings, one at $x$ and the other at $y$ commutes. The sequence of topplings at a single site $x$ do not commute however since the compositions of  $\tau_{x,\rho}$ and $\tau_{x,x+1}$ clearly depends on the order of composition. However given the stacks $\{\xi_{(x,j)}\}$ (see \eqref{stack}) and the function $h_{\alpha}(\cdot)$, one knows in which order the topplings at a single site $x$ occurs.

Before providing a formal proof of the above, we now list some basic facts about the toppling operation in the following lemma.
\begin{lemma}
\label{o:toppling}
The toppling operation as described above has the following properties:
\begin{enumerate}
\item[i.] If $\alpha$ is a legal sequence for $\eta$, then $\Phi_{\alpha}(\eta)$ depends on $\alpha$ only through $h_{\alpha}(\cdot)$. This is immediate from Lemma \ref{abp}.
\item[ii.] $\Phi_{\alpha}(\eta)$ is non-increasing in $h_{\alpha} (x)$ and non-decreasing in $h_{\alpha} (z)$, $z \neq x$. This is easy by observing that a toppling at $x$ can only decrease the number of particles at $x$, whereas a toppling at another site can only increase it.
\item[iii.] If $x$ is unstable in $\eta$ and $\eta_0(x) >\eta(x)$, then $x$ is unstable in $\eta_0$.
\item[iv.] Moreover if $\eta_0 >\eta,$ then $\Phi_{\alpha}(\eta_0) >\Phi_{\alpha}(\eta)$.
\end{enumerate}
Items iii. and iv. above are obvious.
\end{lemma}

We finish this subsection by providing a proof of Lemma \ref{abp}. We note that this is fairly standard, but we however include the short proof for the sake of being self-contained. The same is true for the standard and easy consequences of Abelian property described later: Lemma \ref{abp1} and Lemma \ref{mono1}. We do omit the proof of Lemma \ref{lap1}. For more discussions and details of these arguments in a more general setting, the interested reader is refereed to \cite{bond, RS12}. and the references therein. 
 
\textit{Proof of Lemma \ref{abp}.}   
Let $|\alpha|$ denote the length of the sequence $\alpha.$ Since $h_{\alpha}=h_{\alpha'}$ clearly $|\alpha|=|\alpha'|.$  The proof now follows by induction on $|\alpha|$. Clearly if $|\alpha|=1$, there is nothing to prove. 
Otherwise let $\alpha=(x_1,x_2,\ldots,x_{k})$ and $\alpha'=(y_1,y_2,\ldots y_k).$ Now if $x_1=y_1$ we are done by induction, considering the sequences $(x_2,\ldots x_k)$ and $(y_2,\ldots y_k)$ with initial configuration $\Phi_{x_1}(\eta)=\Phi_{y_1}(\eta).$ If $x_1 \neq y_1$ notice that since $h_{\alpha'}(y_1)\ge 1$ and $h_{\alpha}=h_{\alpha'}$, there exists a $j\le k$ such that $x_j=y_1.$ Let $j_0$ be the minimum such $j$. Now consider the sequence $\alpha_{j_0}=(x_{j_0},x_1,x_2\ldots x_{j_0-1},x_{j_0+1},\ldots,x_k).$ We now notice that $\Phi_{\alpha}(\eta)=\Phi_{\alpha_{j_0}}(\eta)$. To see this first observe that  $\alpha_{j_0}$ is a legal sequence. This is because initially the toppling $\Phi_{x_{j_0}}$ is legal since  $\alpha'$ is legal. Also since $x_{j_0}$ does not occur in $(x_1,\ldots x_{j_0-1})$ by the discussion before the proof, $\Phi_{x_{j_0}}$ commutes with $\Phi_{x_{i}}$ $1\le i \le j_0-1$. Thus $\Phi_{\alpha}(\eta)=\Phi_{\alpha_{j_0}}(\eta)$. Now considering $\alpha_{j_0}$ and $\alpha'$ we are done by the previous case since both the sequences have the same first element. Thus $$\Phi_{\alpha}(\eta)=\Phi_{\alpha_{j_0}}(\eta)=\Phi_{\alpha'}(\eta).$$
\qed

\subsection{Consequences of the Abelian property}\label{finite12}
As mentioned before, for our purposes we will be often interested in finite ARW dynamics restricted to an interval. We start with the following definition.  
\begin{definition}\label{resdef}
Let $V$ be a finite subset of $\Z$. A configuration $\eta$ is said to be stable
in $V$ if all the sites $x \in V$ are stable. We say that a sequence of topplings $\alpha$ is contained in $V$ if all its elements
are in $V$, and we say that $\alpha$ stabilizes $\eta$ in $V$ if every $x \in V$ is stable in $\Phi_{\alpha}(\eta)$.
\end{definition}

\begin{lemma}(Least Action Principle, \cite[Lemma 1]{RS12})\label{lap1}
Given a set $V$, if $\alpha, \beta$ are two legal sequences of topplings such that $\beta$ is contained in $V$ and $\alpha$ stabilizes $\eta$ in $V$ then $h_{\beta}\le h_{\alpha},$ i.e.\ all the topplings in $\beta$ are also needed in $\alpha.$ 
\end{lemma}

As a simple corollary we obtain the following.
\begin{lemma}(\cite[Lemma 2]{RS12})
\label{abp1} 
If $\alpha$ and $\beta$ are both legal toppling sequences for $\eta$ that are contained in $V$ and stabilize $\eta$ in $V$, then $h_{\alpha} = h_{\beta}$ . In particular, $\Phi_{\alpha}(\eta) =\Phi_{\beta}(\eta)$.
\end{lemma}
\begin{proof} By Lemma \ref{lap1}, $h_{\beta}\le h_{\alpha}$ as well as $h_{\alpha}\le h_{\beta}$ and hence we are done.
\end{proof}

We now state a simple but crucial monotonicity result that allows us define $u(0)$ in \eqref{monlim}.
First we define formally the odometer function appearing in \eqref{monlim}.
\begin{definition}(Truncated odometer) Given $\eta \in \Omega,$  a set of vertices  $V$, and any legal stabilizing sequence $\alpha$ as in Lemma \ref{abp1} define the truncated odometer, 
\begin{equation}\label{finalodo1}
u_{V}(\cdot):=u_{\eta,V}(\cdot):=h_{\alpha}(\cdot).
\end{equation}
\end{definition}
 
Note that by Lemma \ref{abp1}, $u_{V}(\cdot)$ is a well defined quantity. We also define the final configuration 
\begin{equation}\label{finalconf1}
\eta_{\infty,V}:=\Phi_{\alpha}(\eta)
\end{equation}
 which is also well defined by the above lemma.
For our purposes we will only consider finite boxes i.e.\ $V=[-M,M],$ for some positive integer $M$. For brevity we define 
\begin{equation}\label{trun1}
u_M:=u_{[-M,M]}.
\end{equation}

\begin{lemma}(Monotonicity, \cite[Lemma 3]{RS12})
\label{mono1} 
For vertex sets $V \subset V'$ and configurations $\eta\le \eta'$, 
$$u_{\eta,V}\le u_{\eta',V'}.$$ 
\end{lemma}
\begin{proof}Let $\alpha$ and $\beta$ be legal stabilizing sequences for the two systems. By iii. and iv. in Lemma \ref{o:toppling} we see that $\alpha$ is also a legal sequence for $\eta'$.  Thus we are done by Lemma \ref{lap1}. 
\end{proof}

Using Lemma \ref{mono1} one can define $u_{\eta,\infty}=\lim_{M\to \infty}u_M.$ Note also that Lemma \ref{mono1} implies that $u_{\eta,\infty}$ does not depend on the increasing sets $[-M,M],$ and any sequence of sets increasing to $\Z$ would have the same limit. 
When the underlying configuration is clear from context we will denote $u_{\eta,\infty}(\cdot)$ just by $u(\cdot)$ as stated in \eqref{monlim}.
A configuration $\eta$ is said to be stabilizable if $u(x)< \infty$ for all $x \in \mathbb{Z}.$ (One can easily see that almost surely $u(\cdot)$ is finite everywhere or infinite everywhere).\\

We now finish off the section with the statement of some technical lemmas and deduce the proof of Theorem \ref{main} from them. 
The rest of this article will be devoted to the proof of the lemmas.
It would also be convenient for us to adopt the standard notational convention throughout the paper of using letters $c,C,$ etc.\ to denote constants whose values are not crucial for the arguments and their values may change from line to line, sometimes within the same proof. 
\subsection{Proof of Theorem \ref{main}}\label{pt2}
Before proceeding further with technical arguments, we now discuss the key steps in the proof of Theorem \ref{main}. By Lemma \ref{equi1} it suffices to show the following: for any $\mu>0$ for small enough sleep rate, almost surely, the sequence $\{u_M(0)\}_{M\in \N}$ (see \eqref{trun1}) takes arbitrarily large values.  To show this we look at the following sequence of intervals of exponentially growing size.
\begin{align*}
V_1&=[0,2^{5}],\\
V_2 &= [0,2^8],\\
\vdots &=\vdots\\
V_{\ell}&=[0,2^{3\ell+2}].
\end{align*}

Recall $\eta$ is the initial configuration on $\Z$ distributed as the product measure $\nu$. Let $\eta^{(\ell)}$ denote the configuration which is the restriction of $\eta$ to the interval $I_{\ell}=[2^{3\ell},3\cdot 2^{3\ell}]$ and zero outside. Observe that $I_{\ell}$'s are just the middle halves of the interval $V_{\ell}$ and are mutually disjoint for $\ell\in \Z$ (see Figure \ref{round1}). Since $\nu$ is a product measure,  $\{\eta^{(\ell)}\}_{\ell\in \N}$ are independent samples from measures $\{\nu_{I_{\ell}}\}_{\ell\in \N}$ respectively where $\nu_{I_{\ell}}$ denotes the coordinate projection of the product measure $\nu$ onto the interval $I_{\ell}$.

We then consider running rounds of truncated ARWs with initial data $\eta^{(\ell)}$ for and vertex set $V_\ell$ for $\ell=1,2,\ldots$. That is, in the $\ell^{th}$ round, we ignore all the particles that are not initially in $I_{\ell}$ and stabilize the initial configuration $\eta^{(\ell)}$ on the interval $V_{\ell}$ using the elements of the stacks (see \eqref{stack}) that have not been used up to $(\ell-1)^{th}$ round. Notice that in this sequential process, we are ignoring certain active particles from previous rounds, however using Lemma \ref{mono1} this will suffice for our purpose of obtaining a lower bound on the odometer counts. Let $w_{\ell}$ denote the increase in the odometer count at $0$ during the $\ell^{th}$ round of the process described above. We shall show that $\sum_{\ell} w_{\ell}$ becomes arbitrarily large. To this end we have the following lemma.

\begin{figure}[hbt]
\centering
\includegraphics[scale=.6]{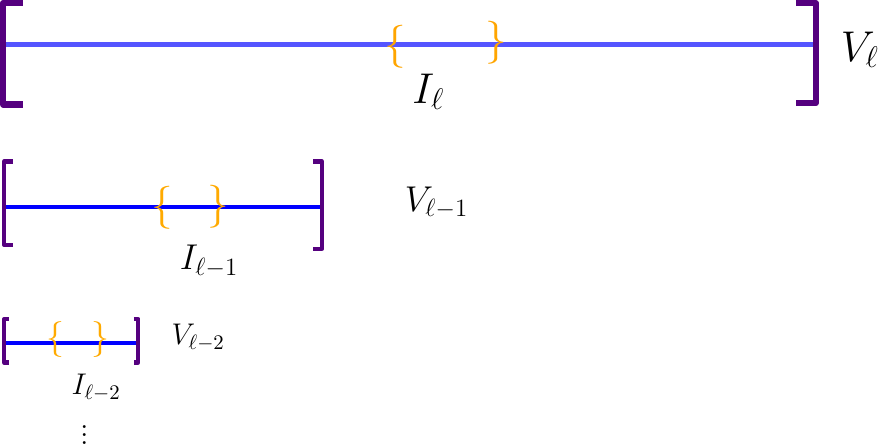}
\caption{The interval of activity $V_{\ell}$ in the $\ell^{th}$ round.}
\label{round1}
\end{figure}

\begin{lemma}
\label{zerohit} 
 $\{w_{\ell}\}_{\ell \in \N}$ are mutually independent. Further, for $\mu>0$ and $\lambda$ sufficiently small depending on $\mu$ we have that for all large $\ell$,
$$\nu\otimes \mathscr{P} (w_{\ell}>0 )\ge \frac{1}{4}.$$
\end{lemma}

We postpone the proof of Lemma \ref{zerohit} for the moment and instead show first how this implies Theorem \ref{main}.

\begin{proof}[Proof of Theorem \ref{main}]
It follows from Lemma \ref{zerohit} that
$\sum_{\ell=1}^{\infty} w_{\ell}=\infty,$ 
$\nu\otimes \mathscr{P}$-almost surely.  
Observe that at the end of the $\ell^{th}$ round, the odometer function at the origin takes value $\sum_{i=1}^{\ell} w_{i}$. Now recall $u_{V_{\ell}}(0)$ from \eqref{finalodo1}. Thus  by Lemma \ref{equi1} we would be done once we show for $\ell\geq 1$, 
\begin{equation}\label{lb23}
u_{V_{\ell}}(0)\ge \sum_{i=1}^{\ell}w_{i}.
\end{equation}
The above is a straightforward consequence of the Abelian Property (Lemma \ref{abp1}) and the monotonicity property (Lemma \ref{mono1}).
Consider the configuration $\eta_{(\rm s)}^{V_{\ell-1}}$ on $V_{\ell-1}$ obtained after stabilizing the particles in $V_{\ell-1}$. 
Now to stabilize the particles in $V_{\ell},$ consider the particle configuration $\eta^*_{\ell}$ which agrees with $\eta_{(\rm s)}^{V_{\ell-1}}$ on $V_{\ell-1}$ and with the original configuration $\eta$ on $V_{\ell}\setminus V_{\ell-1}.$  
Clearly by Lemma \ref{abp1}, $u_{V_{\ell}}(0)\ge u_{V_{\ell-1}}(0)+X$ where $X$ is the total number of topplings at the origin, needed to stabilize the configuration $\eta^*_{\ell}$ in $V_{\ell}$. Now since $I_{\ell}\subset V_{\ell}\setminus V_{\ell-1}, $ trivially $\eta^*_{\ell}\ge \eta^{(\ell)}$ (the restriction of $\eta$ on $I_{\ell}.$) Thus by Lemma \ref{mono1}, $X \ge w_{\ell}$ and the  proof is complete by induction.   
\end{proof}

It remains to prove Lemma \ref{zerohit}. The main step is to show that if the number of particles in the configuration $\eta^{(\ell)}$ is not too small, it is exponentially unlikely in the length of the interval $V_{\ell}$ that during the $\ell^{th}$ round of the operation described above (i.e., stabilizing the initial configuration $\eta^{(\ell)}$ of particles in the interval $V_{\ell}$) none of the particles actually move out of $V_{\ell}$. Recalling the notation for the odometer  from \eqref{finalodo1}, formally we define the event 
\begin{equation}
\A:=\left\{\max\{u_{[-2r,2r]}(-2r),u_{[-2r,2r]}(2r)\}=0\right\}. \label{daejeon}
\end{equation}

In the following lemma we bound the probability of $\A$. Its proof will take the rest of the paper. 

\begin{lemma}
\label{key}
Fix $\mu>0$. Let $\eta_{(r)}$ be a particle distribution supported on $[-r,r]$ such that the total number of active particles in $\eta_{(r)}$ is at least $\mu r$. Let $\sP_{\eta_{(r)}}$ denote the law of the ARW dynamics with initial configuration $\eta_{(r)}$ and the stacks with law $\mathscr{P}$ (as in \eqref{stack}).  
Then for $\lambda$ sufficiently small depending on $\mu$ there exists $c=c(\mu,\lambda)>0,$ such that for all $r$ 
$$\mathscr{P}_{\eta_{(r)}}\left(\A\right)\le e^{-cr}.$$
\end{lemma}

We finish off this section by proving Lemma \ref{zerohit} using Lemma \ref{key}.
\begin{proof}[Proof of Lemma \ref{zerohit}]
Recall that since the intervals $I_{\ell}$ are disjoint, the initial configurations $\eta^{(\ell)}$ are independent samples from the distributions $\nu_{I_{\ell}}$ for $\ell=1,2,\ldots$. Also observe that, because of the independence of the elements of the stack, the joint distribution of the unused elements of the stack at the end of the $(\ell-1)^{th}$ round (i.e.\ the stack that is being used for the $\ell^{th}$ round) is the same as the law $\mathscr{P}$ of the original stacks, and further this is independent of the elements of the stack that have been used in the first $(\ell-1)$ rounds. This implies that $w_{\ell}$ are mutually independent and further that for each $\ell\geq 1$, the distribution of $w_{\ell}$ (under $\nu\otimes \mathscr{P}$) is the same as the law of $u_{V_{\ell}}(0)$ under $\nu_{I_{\ell}}\otimes \mathscr{P}$.    

Let $|\eta^{(\ell)}|$ denote the number of particles in the configuration distributed as $\nu_{I_{\ell}}$. Due to an obvious translation invariance in the system, Lemma \ref{key} implies that,  
$$\nu_{I_{\ell}}\otimes\mathscr{P}\left(\max\{u_{V_{\ell}}(0),u_{V_{\ell}}(2^{3\ell+2})\}=0, |\eta^{(\ell)}|\geq \mu 2^{3\ell}\right)\le e^{-c2^{3\ell+2}}.$$
Observe that by strong law of large numbers for $\ell$ sufficiently large we have $\nu_{I_{\ell}}(|\eta^{(\ell)}|\geq \mu 2^{3\ell})\geq 0.9$ and hence for $\ell$ sufficiently large,
$$\nu_{I_{\ell}}\otimes\mathscr{P}\left(\max\{u_{V_{\ell}}(0),u_{V_{\ell}}(2^{3\ell+2})\}>0\right)\geq 0.9-e^{-c2^{3\ell+2}}\geq \frac{1}{2}.$$

Notice that the interval $I_{\ell}$ and $V_{\ell}$ are both symmetric about $2^{2\ell+1}$. Because $\nu_{I_{\ell}}$ is a product measure with identical marginals, and the dynamics of ARW possesses an inherent left-right symmetry, it follows by a reflection about $2^{3\ell+1}$ that $u_{V_{\ell}}(0)$ and $u_{V_{\ell}}(2^{3\ell+2})$ have identical distributions under the measure $\nu_{I_{\ell}}\otimes\mathscr{P}$ and hence we are done.
\end{proof}

The rest of this paper is devoted to the proof of Lemma \ref{key}. Recall the stacks and the underlying measure $\mathscr{P}$ from \eqref{stack}. In case of finite vertex sets (which will be the focus of our analysis) by Lemma \ref{abp}, the final odometer function and the particle distribution is just a function of the stacks and the initial particle configuration. However for our purposes we need a different set of stacks, labeled particles, and a certain toppling rule. This is described in the next section and we refer to this as \textbf{Labeled ARW dynamics}. 
Thus to distinguish the two, we denote the description of ARW so far (equivalently the measure space corresponding to $\sI$ (see \eqref{stack} and the initial particle configuration $\eta$)  as \textbf{Unlabeled ARW dynamics}\footnote{Note the abuse of notation, since there is no canonical process in the discrete time setting and any order of toppling eventually lead to the same statistics by Lemma \ref{abp}.}.
 The next section also provides a natural coupling between the two `processes'. 

\section{Labeled ARW Dynamics}
\label{ver2}

We start by outlining how we prove Lemma \ref{key}. Remember that the lemma states that if we start with a typical particle distribution in $[-r,r]$ (and 0 outside this interval) then it is unlikely that when we stabilize, the odometer at both $-2r$ and $2r$ is $0$. Typically we shall have about $2\mu r$ many particles in our system and with very high probability we will have at least $\mu r$ particles in the interval $[-r,r]$. Notice that on the event $\A$, which was defined in \eqref{daejeon},
all the particles eventually fall asleep in the interval $[-2r,2r]$. Now consider $K$ a large integer $\gg \mu^{-1}$ (Further discussion on the choice of $K$ is presented in \S~\ref{s:proofkey}). Divide the interval $[-2r,2r]$ into subintervals $[iK,(i+1)K]$ of length $K$ each. In the final configuration, we typically expect only a small number of particles asleep on each interval $[iK,(i+1)K]$ if $\lambda$ is sufficiently small depending on $K$. Indeed, whenever a particle is released from the site $iK$ it is extremely likely (by taking $\lambda$ small) that this particle will reach either $(i-1)K$ or $(i+1)K$ without falling asleep in between. Further, in such a case, it wakes up all the sleepy particles in either $[iK,(i+1)K]$ or $[(i-1)K,iK]$. As $\lambda$ is small, it is highly likely that all of these recently woken particles will make it to either $(i-1)K$ or $(i+1)K$ before falling asleep again. Hence  it is unlikely that  in the $4r/K$ many intervals of length $K$ ,on average there will be $K\mu /4$ particles asleep (since $K \gg \mu^{-1}$). Combining the above, we conclude that since there are at  least $\mu r$ many particles to begin with,  at least one particle makes it to $-2r$ or $2r$.

To make this intuition into a proof, we shall introduce a different set of stacks of instructions, where, instead of a single step, the instructions now  consist of a random walk path that tells the particle (started at $iK$) its steps until it reaches $iK\pm K$. 

\subsection{Random Walk Stacks}
Let  $\lambda>0$ and consider the  lazy symmetric random walk $S_{\lambda}(\cdot)$ with laziness $\frac{\lambda}{\lambda+1},$ on $\Z$ started from the origin. By this we mean $S_{\lambda}(0)=0$ and for each discrete time step, $t=0,1,2,\ldots$, given the path up to time $t,$ 
\begin{align}\label{steps1}
S_{\lambda}(t+1)-S_{\lambda}(t)=
\left\{
\begin{array}{cc}
0 & \text{with probability } \frac{\lambda}{\lambda+1}\\
1 & \text{with probability } \frac{1}{2(\lambda+1)}\\
-1 & \text{with probability } \frac{1}{2(\lambda+1)}.\\
\end{array}
\right.
\end{align}
 
Now for any positive integer $K,$ let $\tau(\{-K,K\})$ denote the hitting  time of the set $\{-K,K\}$ for the walk $S_{\lambda}(\cdot)$ and let
$$S_{K,\lambda}(\cdot)=S_{\lambda}(\min(\cdot, \tau(\{-K,K\}))),$$  be the killed random walk, i.e.\ $S_{\lambda}(\cdot)$ stopped on hitting $\{-K,K\}$. 

We now define a probability space where instructions will be the killed random walk paths instead of single steps as in \eqref{stack}. 
We consider the renormalized lattice, 
\begin{equation}\label{renorlatt}
K\Z=\ldots -2K,-K,0,K,-2K, \ldots,
\end{equation}
and the following set of instructions at the points $K\Z$ (which we shall sometimes refer to as lattice points): 
\begin{equation}\label{stack2}
\mathscr{I_*}=
\begin{array}{ccccccc}
\ldots&\zeta_{(-2,1)}&\zeta_{(-1,1)}&\zeta_{(0,1)}&\zeta_{(1,1)}&\zeta_{(2,1)}&\ldots\\
\ldots&\zeta_{(-2,2)}&\zeta_{(-1,2)}&\zeta_{(0,2)}&\zeta_{(1,2)}&\zeta_{(2,2)}&\ldots\\
\vdots&\vdots&\vdots&\vdots&\vdots&\vdots&\vdots
\end{array}
\end{equation}
where $\zeta_{(x,j)}$ are independent for any $x \in \Z$ and $j\in \N,$ and moreover, 
\begin{equation}\label{law2}
\zeta_{(x,j)}\overset{d}{=}xK+S_{K,\lambda}(\cdot),
\end{equation}
i.e.,\ the elements $\zeta_{(x,j)}$ are i.i.d.\ copies of random walk paths $S_{K,\lambda}(\cdot)$ started at $xK$ and stopped on hitting $xK\pm K.$
We will still continue to think of the random walk steps to be distributed as the elements $\xi_{(\cdot,\cdot)}$ in \eqref{stack} namely:
\begin{equation}\label{law3}
 \begin{array}{cc}
\tau_{z,z-1} & \text{with probability } \frac{1}{2(\lambda+1)}\\
\tau_{z,z+1} & \text{with probability } \frac{1}{2(\lambda+1)}\\
\tau_{z,\rho} & \text{with probability } \frac{\lambda}{\lambda+1}.
\end{array}
\end{equation}
where $z$ is the current location of the random walk.
That is left and right steps in the random walks would be interpreted as instruction to go left or right, whereas lazy steps would be interpreted as sleep instructions. However, how these instructions act upon the configurations will be slightly different from the standard ARW dynamics of \S~\ref{des1} as explained below. We call the above product measure $\mathscr{P}_*$ analogous to $\mathscr{P}$.

\subsection{Labeled Particles and Labeled ARW Dynamics}\label{label23}
With the above stacks we now define a variant of the ARW dynamics from the discussion following Definition \ref{stable1}. To define the modified dynamics, we shall restrict ourselves to initial particle configurations supported on $K\Z$.  In this setting, the particles will be labeled. The labels will keep changing over time and will be from the set $\Z \times \N.$

\textbf{Labeling Scheme:}
Suppose we start with a configuration $\eta_*$ supported on $K\Z$. The first co-ordinate of the label of all the particles starting at $iK$ will be $i$. The second label is an arbitrary chosen  bijection to $\llbracket1, n_i\rrbracket$ where $n_{i}$ is the number of particles at $iK$ initially, i.e., the particles are labeled $(i,1), (i,2), \ldots , (i,n_i)$ in an arbitrary way. See Figure \ref{initial}.  
\begin{figure}[hbt]
\centering
\includegraphics[scale=.6]{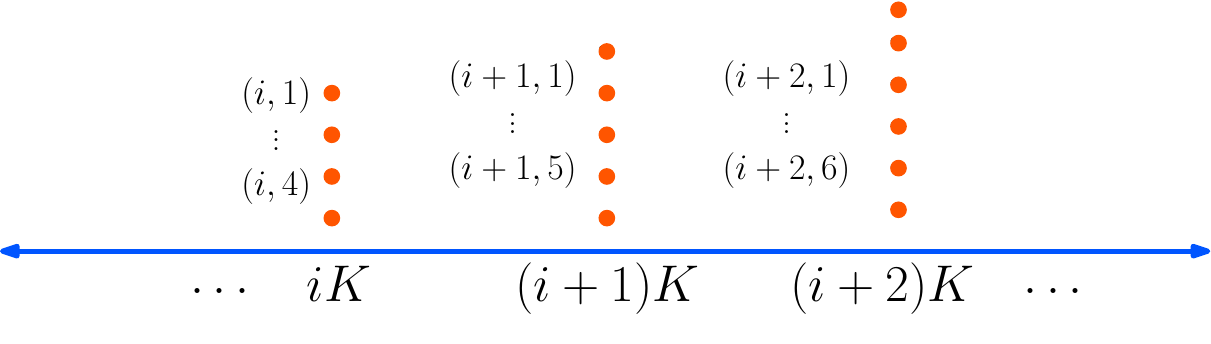}
\caption{Initial labeling}
\label{initial}
\end{figure}

As mentioned above, the labels are not static but evolve with time. Any particle of label $(i,j)$, will change its label when it hits the lattice points $(i-1)K$ or $(i+1)K,$ in which case the label would change to $(i',j')$ where $i'=i\pm1$ according to the lattice point hit and $j'$ is the smallest positive integer such that the label $(i',j')$ has never been used before in the history of the process. A particle with label $(i,\cdot)$ will sometimes be referred to as a particle \textbf{emitted} from $iK$.

We now introduce a toppling scheme for the labeled process analogous to \S~\ref{s:ver1}. To do that we first associate with each particle a random walk from the stacks \eqref{stack2}.

\textbf{Associating Random Walks with Particles:}
The first time a particle acquires label $(a,b)$ (recall that our labeling rule implies that at this time the particle must be at $aK$), it gets associated with the random walk $\zeta_{(a,b)}$ which stays associated with the particle until it changes its label next. 

Next we define how to use the random walks to topple particles. Informally we do the following. Whenever we decide to topple a particle with label $(a,b)$, we use the first unused step of the random walk $\zeta_{(a,b)}$. Let $\tau$ be the total number of steps in $\zeta_{(a,b)}$ (i.e., until the random walk hits $aK\pm K$). Whenever the particle is toppled, before $\tau$, say for the $i^{th}$ time, the toppling step is the $i^{th}$ step of  $\zeta_{(a,b)}$. We now provide a formal description.

\textbf{Toppling Rule and Labeled ARW dynamics:}
Recall the transitions $A\to S$ and $A+S\to 2A$ between particles in active and sleepy states from \S~\ref{s:ver1} and that only $A$ particles could be toppled. We introduce a particle interaction which is a slight variant of that dynamics, where the particles with different first labels do not interact. We shall call this variant, the Labeled ARW dynamics and it will always be defined on a finite universe $V$. That is, we shall fix a finite subset $V$ of $\Z$ and and particles outside $V$ become instantaneously inactive and stay inactive for ever. To distinguish from the standard ARW dynamics from \S~\ref{s:ver1}, we denote the state of active and sleepy particles in the Labeled ARW dynamics by $\tilde{A}$ and $\tilde{S}$ respectively.  

Consider stacks of random walks distributed according to $\mathscr{P}_{*}$ (defined in the previous section) and a particle configuration $\eta_{*}$ distributed supported on $K\Z$. Fix $V$ and let all particles in $V$ initially be in $\tilde{A}$ state. Consider the labeling scheme described above. For concreteness  we shall work with a particular order of topplings. To this end we start by defining the lexicographic order on the labels: i.e.\  $(a,b)<(c,d)$ if $a<b$ or $a=b$ and $c<d$.

At each time $t$, we topple the $\tilde{A}$ particle with the smallest label, say $(a,b)$, which is well defined as we are restricted to the finite set $V$.\footnote{Note that in Unlabeled ARW dynamics, particles were unlabeled and hence the definition of the process involved toppling sites rather than particles unlike here.} To topple, use the first unused instruction of random walk $\zeta_{(a,b)}$. The left and right step instructions from the random walks are interpreted as before and used to move the particle one step to left or right respectively. The lazy steps in the random walks are interpreted as sleep instructions. 

However, in the Labeled ARW dynamics the
$\mathbf{\tilde{A}+\tilde{S}\to 2\tilde{A}}$ transition occurs only when the label of the two particles have same first co-ordinate. Let us illustrate this with an example. Consider an $\tilde{S}$ particle at $y\in(aK, (a+1)K)$ with label $(x,j)$. Notice that by definition either $x=a$ or $x=a+1$. Now if an $\tilde{A}$ particle is toppled from $y\pm 1$ and moves to $y$, the transition $\tilde{A}+\tilde{S}\to 2\tilde{A}$ occurs only if the $\tilde{A}$ particle has label of the form $(x,\cdot)$. 
Thus in words the above says that a particle which was emitted from a lattice site $aK\in K\Z$ can be woken up only by another particle emitted from $aK$. 

Also in the Labeled ARW dynamics the
$\mathbf{\tilde{A}\to \tilde{S}}$ transition is allowed to happen only when the label of no other particle at that site has the same first co-ordinate. Notice the difference with $A\to S$ transition. As observed before, the particles with different first labels do not interact. \footnote{As a consequence, it is allowed to have two $\tilde{S}$ particle at a  site, provided they have different first labels. Moreover, if  there are two $\tilde{S}$ particles at $y\in(iK, (i+1)K)$ (one with first label $i$ and another with $(i+1)$) and an $\tilde{A}$-particle with first label $i$ moves to $y$. It then wakes up the $\tilde{S}$-particle at $y$ with label $i$, but not the one with label $i+1$, i.e., in such a situation we have a $\tilde{A}+2\tilde{S}\to 2\tilde{A}+\tilde{S}$ transition.      
}

The system is said to stabilize if all the particles in $V$ are in $\tilde S$ state (recall that a particle that moves out of $V$ is inactive forever). As before it is easy to argue that as $V$ is finite and the number of particles is finite the system stabilizes almost surely in finite time. We define the odometer function and the final configuration analogous to \eqref{finalodo1} and \eqref{finalconf1} respectively. 

Note in Unlabeled ARW dynamics there was no toppling rule. Hence to ensure that \eqref{finalodo1} and \eqref{finalconf1} are well defined we had to invoke the Abelian property (Lemma \ref{abp}).
However in the Labeled ARW dynamics the toppling rule is fixed and hence there is no ambiguity in the definitions of the final odometer count and the final configuration.  We denote the odometer function and the final configuration by 
\begin{equation}\label{labelnot1}
u_{*,V}(\cdot) \mbox{ and }\eta_{*,\infty,V}
\end{equation}
 respectively. Note that the dependence on the initial configuration is suppressed.

In the next section we relate the Labeled ARW dynamics with the Unlabeled ARW dynamics. Since by definition in the labeled process, particles with different first labels do not interact, one would expect more active to sleepy transitions and less sleepy to active transitions, than the unlabeled process.  So started with the same initial configuration it is expected that the Labeled ARW dynamics is quicker to stabilize (in a stochastic sense).  To make this formal we need to define a coupling between the measures $\sP$ and $\sP_{*}$ (the law of the stacks used in the unlabeled and the Labeled ARW dynamics respectively, see \eqref{stack}, \eqref{stack2}).

\subsection{Coupling Labeled ARW Dynamics with Unlabeled ARW dynamics}
Fix a finite subset $V\subseteq \Z$. Given an initial particle configuration $\eta$ supported on $V$, we define a natural coupling $\cP:=\cP_{\eta}$ of the stacks $\mathscr{I}$ and $\mathscr{I}_{*}$ having laws $\sP$ and $\sP_{*}$ respectively. The coupling is run in three rounds.  Note that the labeled process is only defined when the initial configuration is supported on $K\Z.$ This is what the first round achieves.\\

\textbf{First Round:}
We run Unlabeled ARW dynamics in $V$ where we topple all active particles in $\Z\setminus K\Z$. After the end of this round the system only has active particles at $V\cap K\Z$. Call the random configuration of active particles $\tilde{\eta}$, supported on $K\Z$. The distribution of $\tilde{\eta}$ is a function of $\eta$. 

\textbf{Second Round:}
In this round we describe a `natural' coupling  of Unlabeled ARW dynamics and the Labeled ARW dynamics starting from $\tilde{\eta}$. The coupling $\mathcal{P}$ will be Markovian.
Recall the labeling of the particles and toppling rule described in \S~\ref{label23}. 
Also recall that by the Abelian property (Lemma \ref{abp}) the final statistics of Unlabeled ARW dynamics are independent of the order in which particles or sites are toppled.  Thus in order to couple with the labeled process we choose the same toppling rule as the latter. Namely, we topple the same particle in both the processes and couple the stacks $\sI$ and $\sI_*$ such that the two processes stay identical throughout. Note that this is possible due to the following two reasons:
\begin{itemize}
\item
The steps of the random walk $\zeta_{(\cdot,\cdot)}$ in \eqref{stack2} have the same distribution as $\xi_{(\cdot,\cdot)}$ in \eqref{stack}, as already remarked in \eqref{law3}.
\item
Arguing inductively, if the particle configuration stays the same in both the versions under $\mathcal{P}$ up to the  $t^{th}$ toppling, since an $\tilde{A}$ particle is always an $A$ particle, the $(t+1)^{th}$ toppling would be legal in both the processes. 
\end{itemize}
As a consequence of the above, the stabilizing sequence in the Labeled ARW dynamics would always be a legal sequence in Unlabeled ARW dynamics. Thus this round of the coupling runs until the Labeled ARW dynamics has been stabilized in $V$.

Note that potentially many of the particles in $\eta$  that were not in $\tilde\eta$ and hence not toppled in the second round could have woken up in the process rendering the configuration running in Unlabeled ARW dynamics unstable. The last round takes care of these remaining particles.

\textbf{Third round:}
Independently sample the rest of the stacks $\sI$ and $\sI_*$. Use the former and run Unlabeled ARW dynamics until stabilization, toppling arbitrarily. By Lemma \ref{abp}, this does not change the final distribution. 

Thus we have completed the description of the coupling between the two processes. The following lemma follows directly from the  definition of $\cP_{\eta}$.

\begin{lemma}
\label{o:ver1comp}
Under the coupling $\cP_{\eta}$ described above, $u_{V}(\cdot)$ started from $\eta$ is lower bounded by $u_{*,V}(\cdot)$ (the odometer for the labelled process, see \eqref{labelnot1}) starting from $\tilde{\eta}$.
\end{lemma}

To prove Lemma \ref{key}, it suffices to lower bound $u_{*,[-2r,2r]}(\cdot)$ started from $\tilde{\eta}$. Recall that $\tilde{\eta}$ is a random particle configuration supported on $K\Z$ whose distribution depends on $\eta$, but is independent of the stack $\mathscr{I}_{*}$ and $\eta$ is an initial configuration supported on $[-r,r]$ with at least $\mu r$ particles. This is the content of the next two lemmas. 

\begin{lemma}
\label{key1} 
Fix $\mu >0$. There exists $K=K(\mu)$ and $\lambda=\lambda(\mu,K)>0,$ such that the following is true for all  large enough $r$.
Consider an initial particle configuration $\widehat{\eta}$ of active particles supported on $[-r,r]\cap K\Z$ with $\frac{\mu r}{2} \le |\widehat{\eta}|\le 2\mu r$ particles. Then there exists $c>0$ such that for all $r$
\begin{equation}\label{keyalt}
\mathscr{P}_{*,\widehat{\eta}} (\A)\le e^{-cr}
\end{equation}
where $\sP_{*,\widehat{\eta}}$ denotes the law of the Labeled ARW dynamics started with the initial particle configuration $\widehat{\eta}$ and $\A$ was defined in \eqref{daejeon}.
\end{lemma}

\begin{lemma}
\label{l:particle}
Recall the coupling $\cP_{\eta}$ from above. For any $\eta$ supported on $[-r,r]$ with at least $\mu r $ many particles, for $\lambda$ sufficiently small there exist $c>0$ such that  
$$\cP_{\eta}\left(|\tilde{\eta}|<\frac{\mu r}{2}\right)\leq e^{-cr}$$
where $|\tilde{\eta}|$ denote the number of particles in $\tilde{\eta}$.
\end{lemma}

\begin{proof}
For any particle in $\eta$, at $y\in \big((x-1)K,xK\big)$, notice that for $\lambda$ sufficiently small depending on $K$, each particle has a chance at least $\frac{3}{4}$ of reaching $\{(x-1)K, xK\}$ before falling asleep and these events are independent of each other. Since all such particles are in $\tilde{\eta}$, the conclusion now follows from the fact that $|\eta|\ge \mu r$ and a large deviation estimate. 
\end{proof}  

We finish off this Section with the proof of Lemma \ref{key}. The rest of this paper is devoted to proving Lemma \ref{key1}.

\begin{proof}[Proof of Lemma \ref{key}]
Given $\mu>0$ we get $K$ and $\lambda$ from Lemma \ref{key1}.
The proof follows from a combination of Lemmas \ref{l:particle},  \ref{key1} and  \ref{o:ver1comp}.
\end{proof}

In the following section we develop the remaining technical results needed for the proof of Lemma \ref{key1}.
\section{Technical Preliminaries}
\label{s:prelim}
Recall the set up of Lemma \ref{key1}: fix $K=K(\mu)$ sufficiently large and $r$ sufficiently large depending on $K$ and a particle configuration $\widehat{\eta}$ supported on $[-r,r]\cap K\Z$ that contains at least $\frac{\mu r}{2}$ many particles. Consider Labeled ARW dynamics on $[-2r,2r]$ starting with initial configuration $\widehat{\eta}$.  We shall denote the law of this process by $\sP_{*,\widehat{\eta}}$. Let $\widehat{\eta}_{\infty}$ denote the final configuration and for $xK\in (-2r,2r),$ let $\widehat{\eta}_{x,\infty}$ denote the final configuration of the particles labeled $(x,\cdot)$. By definition, it follows that $\widehat{\eta}_{x,\infty}$ is supported on $((x-1)K, (x+1)K)$.
Thus as before $|\widehat{\eta}_{x,\infty}|$ denotes the total number of particles  labeled $(x,\cdot)$ in $\widehat{\eta}_{x,\infty}$ (note that all of the these particles are sleepy.). Moreover let $|\widehat{\eta}_{x,\infty,-}|$  denote the total number of particles (sleepy) labeled $(x,\cdot)$ in the interval $((x-1)K,xK)$ and similarly let $|\widehat{\eta}_{x,\infty,+}|$  denote the total number of particles (sleepy) labeled $(x,\cdot)$ in the interval $(xK,(x+1)K).$
Also recall that $u_{*,[-2r,2r]}(x)$ counts the number of times a labeled particle is toppled at a site $x$,  up to stabilization. However for our purposes we need to determine how many elements (random walk paths) from the stacks in \eqref{stack2} have been used. For $xK\in K\Z$, define the \textbf{renormalized odometer} as: 
\begin{equation}
\label{e:reodo}
M(x):= \max \left\{ j: \mbox{ a particle labeled }(x,j) \mbox { was emitted}\right\},
\end{equation}
i.e.\ the number of random walk paths $\zeta_{(x,\cdot)}$ in \eqref{stack2} that were used until stabilization.

Now we define certain statistics of the process, which will be used in the proof of Lemma \ref{key1}.  For each $xK\in [-2r,2r]$, we keep track of the net number of particles labeled $(x,\cdot)$ (emitted from $xK$)  hitting $(x-1)K$, $xK-1$, $xK+1$ and $(x+1)K$ respectively by 

\begin{align}\label{not21}
L(x-1 \leftarrow x)&:=L(x-1 \leftarrow x)(\widehat{\eta}), \\ 
\nonumber
L(\leftarrow x )&:=L(\leftarrow x)(\widehat{\eta}), \\
\nonumber
R(x \to )& :=R(x \to) (\widehat{\eta}), \\
\nonumber
R(x \to x+1)& :=R(x \to x+1)(\widehat{\eta})
\end{align}
respectively. 
More formally, since $M(x)$ is the total number of random walk instructions $\zeta_{(x,\cdot)}$ that have been used,   
we have \begin{align*}L(x-1 \leftarrow x)& =\sum_{i=1}^{M(x)}\mathbf{1}(\text{ the particle labelled }(x,i) \text{ hit }(x-1)K),\\
L(\leftarrow x )& =L(x-1 \leftarrow x) + |\widehat{\eta}_{x,\infty,-}|.
\end{align*}
Thus,  $L(x-1 \leftarrow x)$ is the number of random walk paths $\zeta_{(x,i)}$ where $1\le i \le M(x),$ that have hit $(x-1)K,$ whereas  $L(\leftarrow x)$ is the total number of random walk paths that have either settled down in the interval $((x-1)K, xK),$ or have hit $(x-1)K.$

It is important to note that we are counting the net number of particles in all the counts, i.e., if a random walk path moved back and forth across $xK$ a few times to eventually land somewhere in the interval $((x-1)K,xK)$, it just increases $L(\leftarrow x)$  by one and does not change the other quantities;  if it hits $(x-1)K$ it increases both  $L((x-1) \leftarrow x)$ and $L(\leftarrow x)$ by one each but do not change $R(x \to x+1)$ or $R(x \to).$ 

Similarly $R(x \to x+1)$ and $R(x \to),$  are defined. We omit the details. 
 See Figure \ref{hneigh} for an illustration.
\begin{figure}[hbt]
\centering
\includegraphics[scale=.7]{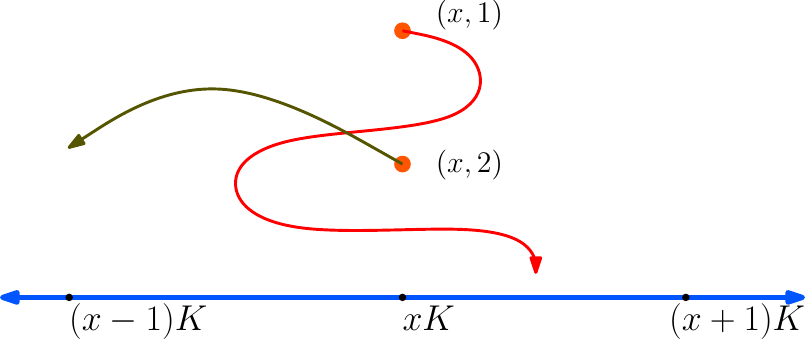}
\caption{The particle $(x,1)$ eventually lands to the right of $xK$ and hence contributes one to $R(x \to ).$ On the other hand the particle $(x,2)$ hits $(x-1)K$ and hence contributes one to $L(\leftarrow x)$ and $L(x-1\leftarrow x)$. } 
\label{hneigh}
\end{figure}
Next we observe that the statistics in \eqref{not21} are ``local" functions of the Labeled ARW dynamics which is precisely the purpose behind defining the latter, i.e., particles with different first labels do not interact. We make things formal below.
\subsection{Single site dynamics}\label{ssd}
 For any $x\in \Z,$ consider the Labeled ARW dynamics in $V_x=((x-1)K,(x+1)K)$ started with $m$ particles at $x$. Then the quantities analogous to \eqref{not21}  will be denoted by 
\begin{equation}
\label{ssd1}
L^{SS}_{x-1 \leftarrow x }(m), L^{SS}_{ \leftarrow x }(m), R^{SS}_{x \to}(m), R^{SS}_{x \to x+1}(m)
\end{equation}
respectively. Thus the above can be thought as a localized version of the Labeled ARW  dynamics. The final configuration in this process will be denoted by $\eta^{SS}_{x}(m)$.

\vspace{0.5cm}

Consider Labeled ARW dynamics started with initial configuration $\hat{\eta},$ stabilized in $(-2r,2r)$. The particles with label $(x,\cdot)$ are toppled only at certain  times say $t_1, t_2,\ldots$. Thus the quantities in \eqref{not21} are functions of the configurations at these times. By definition $$\{(x,1),(x,2),\cdots , (x,M(x))\}$$ are the different labels of such particles where $M(x)$ is the renormalized odometer defined in \eqref{e:reodo}. 

By definition,  all the elements/random walks $\zeta_{(x,i)}$ for $i \le M(x)$ are used, and this by definition happens in one of two scenarios, either a particle started at $x$ and hence used a stack element $\zeta_{(x,\cdot)}$ when it was first  toppled, or a particle labeled $(x+1,j)$ or $(x-1,j)$ for some $j$ hits $x$ and changes its first label to $x$ and uses a stack element $\zeta_{(x,\cdot)}.$

Hence we have the following lemma whose proof follows by an easy induction on $M(x)$ and is omitted.
\begin{lemma} 
\label{crucial}
For a fixed initial configuration $\hat{\eta}$ recall the definition of $M(\cdot)$ given in \eqref{e:reodo}. For each $x\in \Z$ such that $xK\in (-2r,2r)$, the quantities in \eqref{not21} are the same as the respective quantities in \eqref{ssd1} with $m=M(x)$, where both dynamics use the same stack $\mathscr{I}_{*}$ (see \eqref{stack2}).  Moreover  $\eta^{SS}_{x}(M(x))=\widehat{\eta}_{x,\infty}$, i.e., the eventual distribution of the particles labeled $(x,\cdot)$ is the same in both processes.
\end{lemma}

Notice that the single site process clearly is just a function of $M(x)$ and the stack at $x$ i.e.,  $\left(\{\zeta_{(x,\cdot)}\}\right)$. Hence the above observation says that after the labeled system has fixated, the distribution of the particles labeled $(x,\cdot),$ and the other local statistics defined in \eqref{not21}, are just functions of the stack elements $\{\zeta_{(x,j)}:1\le j \le M(x)\}$ where $M(x)$  is the renormalized odometer value at $x$.

\section{Proof of Lemma \ref{key1}}
\label{s:proofkey}

 We start by describing the rough idea before providing formal arguments. 
Recall the event from \eqref{daejeon}, 
\begin{equation}\label{event12}
\A:=\left\{\max\{u_{*,[-2r,2r]}(-2r),u_{*,[-2r,2r]}(2r)\}=0\right\},
\end{equation}
and that the initial configuration $\hat{\eta}$ is supported on $[-r,r]\cap K \Z$ with $2\mu r\geq  |\hat{\eta}|\geq \frac{\mu r}{2}$. 

We  have by definition: 
\begin{lemma}\label{totalsleep}If $|\hat \eta|\ge \frac{\mu r }{2},$ on the event  $\A$, no particle hits $\{-2r,2r\},$ and hence at least $\frac{\mu r}{2}$ particles fall asleep inside $(-2r,2r).$
\end{lemma}

We shall need the following notation to keep track of the net flux of particles across lattice sites.  Namely for every $x \in 
\mathbb{Z},$ let 
\begin{align}\label{fluxnot1}
F^{+}_{x}&:= R(x \to) -L(x \leftarrow x+1) \\
\nonumber
F^{-}_{x} &:= L(\leftarrow x) -R(x-1 \rightarrow x). 
\end{align}
where the quantities on the RHS were defined in \eqref{not21}.
 Thus $F^{-}_{x}$ and $F^{+}_{x}$ denote  the net number of particles moving across the half-integer points  $xK-\frac{1}{2}$ and $xK+\frac{1}{2}$ respectively.  Since we assume that the total number of particles initially, i.e., $|\hat \eta|$ is at most $2\mu r$, it follows that $|F^+_{x}|$ and $|F^-_x|$ are both bounded by $2\mu r.$
Recall the renormalized odometer function $M(\cdot)$ from \eqref{e:reodo}. 
For a positive integer $a_0,$ and integers $f^{+}_0, f^{-}_0$, let $\mathcal{D}=\mathcal{D}(a_0,f^{+}_0,f^{-}_0)$ denote the event that the following equalities hold:

\begin{align}\label{initdata1}
M(0) =a_0,  F^{+}_{0} =f^{+}_0, F^{-}_{0} =f^{-}_0.
\end{align}

For a fixed triple $(a_0,f^{+}_0,f^{-}_0),$ we shall show that $\mathcal{D}\cap \A$ is unlikely enough such that Lemma \ref{key1} will then follow from a union bound over all possible values of different triples $(a_0,f^{+}_0,f^{-}_0)$. It turns out that we need to consider two cases, we start with the following case, where $M(0)=a_0$ is `small', i.e., $a_0<r^6$.  
The case when $M(0)>r^6$  is easy and follows from well known random walk estimates. 

The analysis of the first case is the most technically difficult part of this paper. Hence for the ease of reading we provide a detailed roadmap of the proof along with what various subsequent lemmas achieve.

Recall from \S~\ref{sop} about the general intuition behind the proof: Since the number of particles starting in $[-r,r]$ is assumed to be at least $\frac{\mu r}{2},$ to show that after the stabilization process, particles exit the interval $[-2r,2r]$ with high probability, it suffices to show that there are much less than $\frac{\mu r}{2},$ sleepy particles in $[-2r,2r].$ 
We consider the lattice points $K\Z=\{\ldots,-2K,-K,0,K,2K,\ldots\},$ for $K\gg \frac{1}{\mu}$ and then show at the end of the stabilization process, for $\lambda$ small enough, it is unlikely to have more than a particle asleep in any interval of length $K$. 
The reason being that, choosing $\lambda$ small enough makes it very unlikely for a particle to fall asleep before moving from a multiple of $K$ to another. 

Now it is easy to make the above intuition formal if the odometer function $M(\cdot)$ was specified. However for the purposes of our proof we have to do a careful union bound over various possible odometer values, and hence for any specific sequence of odometer values, we need to have  sharp enough probability upper bounds  on the event that there are many sleepy particles in the interval $[-2r,2r],$ which would enable us to carry out such a strategy. 
Below we give a step by step description of how our proof achieves the above.  In what follows we only control the sites on the positive integer axis and the negative axis can be handled similarly.

\begin{enumerate}
\item Given the value of the normalized odometer $M(0)<r^6$ at the origin and the flux $F_0^{+}=f_0^+$, (the net number of particles crossing the point $\frac{1}{2}$ in the positive direction), we make an estimation of $M(1)$ in the following way:  Using the single site dynamics at $0$ and the stack of random walks $(\zeta_{(0,i)})_{i=1}^{M(0)},$ we can compute $R_{0\to}^{SS}(M(0)),$ the net number of particles labelled $(0,\cdot)$ crossed $1/2$ which along with $f_0^+$ tells us exactly the value $\ell_1,$ of $L_{0\leftarrow 1}^{SS}(M(1)),$ the number of particles labelled $(1,\cdot)$ which hits $0.$
 \item We now look at  $L_{0\leftarrow 1}^{SS}(\cdot)$ as a function of the number of labelled particles emitted from the site $K.$ Clearly the function is monotone increasing (not strictly) with jump sizes bounded by $2K$ (this is proved in Lemma \ref{monojump}). Thus this creates the possibility,  that the function $L_{0\leftarrow 1}^{SS}(\cdot)$ never attains the value $\ell_1$. This suggests that our assumption on $M(0)$ and $F_0^+$ was not correct. 
 \item However on the other hand it might also happen that the function $L_{0\leftarrow 1}^{SS}(\cdot)$ attains the value $\ell_1$ on an entire interval $[M_*(1),  M_*(1)+w(1)].$ Now if $w(1)$ is zero, this implies that $M(1)=M_*(1).$ However if $w(1)>0,$ then there is no clear choice for $M(1).$ 
\item Suppose we made an arbitrary choice for $M(1)\in [M_*(1),  M_*(1)+w(1)]$. Using the stack at $K,$ now one can exactly compute the value $f_1^+$ of the flux $F^+_1,$ (see \eqref{ind2}) which then helps us compute again as above the value $\ell_2$ of  $L_{1\leftarrow 2}^{SS}(M(2))$.
\item We repeat the same argument as above to see that the value $\ell_2$ is either not attained or potentially attained on an entire interval  $[M_*(2),  M_*(2)+w(2)]$ and make an arbitrary choice of $M(2).$ Note that the interval $[M_*(2),  M_*(2)+w(2)]$ depends on our potentially arbitrary choice of $M(1)$ in the previous step. 
\item To keep our track of our string of arbitrary choices one dependent on the previous ones we we introduce a parameter $\underline{y}=\left\{ y(0)=0,y(1),y(2), \ldots , y(\frac{2r}{K}-1)\right\}.$ which is a sequence of non-negative integers (see \eqref{treepath}). The goal is to choose $M(i)=M_*(i)+y(i).$
Obviously as mentioned above, the goal can fail in a variety of ways, either $M_*(i)$ does not exist or $y(i)>w(i).$
\item We say a vector $\underline y$ is $\textbf{realizable}$, if all the above choices succeed. Note that whether $\underline y$ is $\textbf{realizable}$ or not depends on the underlying randomness in the stack (see \eqref{stack2}).
\item Note that given the stack of random walks and a $\textbf{realizable }\underline y,$ we can compute the number of particles which are asleep in the interval $[0,2r].$
\item Taking care of the negative interval $[-2r,0]$ exactly as above, we then prove the following key upper bound, (see \eqref{expmoment}), that 
\begin{equation}\label{prelimexp23}
\E \left(e^{\sum_{i=-2r/K+1}^{2r/K-1}[\tilde S_{i,i+1}+\tilde S_{i,i-1}]}\mathbf{1}[\mathcal{R}(\underline{y})\cap \mathcal{D}]\right)\le e^{2K}(.52)^{{\sum_{i=-2r/K+1}^{2r/K-1}y_i}},
\end{equation}
 where $\mathcal{R}(\underline{y})$ denotes the event that  $\underline y$ is $\textbf{realizable},$ and $\cD$ was defined above in \eqref{initdata1}.
The proof of this is achieved using Lemma \ref{probbound675}.

\item Assuming that the initial number of particles in $[-2r,2r]$ is at least $\frac{\mu r}{2},$ as mentioned in Lemma \ref{totalsleep}, on the event $\cA_r,$, we have
$$\sum_{i=-2r/K+1}^{2r/K-1}[\tilde S_{i,i+1}+\tilde S_{i,i-1}]\ge \frac{\mu r}{2}.$$ 
\item The proof then is complete by Markov inequality (Lemma \ref{quantbnd1}), and then summing the above bound over all possible choice of $\underline y,$ (Lemma \ref{l:smallodo}) followed by summing over all possible choices of $M(0)<r^6$ and the flux values $F^+_0$ and $F^-_0,$ both of which takes values bounded by the total number of particles initially on $[-2r,2r]$ which is by hypothesis at most $2\mu r.$ Note also that the number of terms in the sum in the exponent in \eqref{prelimexp23}, is $O(\frac{r}{K}).$ Thus by choosing $K$ large enough depending on $\mu,$ we can ensure that Markov's inequality indeed yields a bound of the form $e^{-cr}$ for some $c>0$ for all large $r.$
\item Finally we address the case $M(0)> r^6$ by showing that the odometer of the actual ARW dynamics is itself smaller than $r^{6}$ with probability at least $1-e^{-cr}$ (Lemma \ref{l:largeodo}). Now the renormalized odometer $M(0)$ is upper bounded by the odometer of the actual ARW dynamics (see Lemma \ref{o:ver1comp}) and hence we are done.  
\end{enumerate}

\subsection{$M(0)$ is small.}

\begin{lemma}
\label{l:smallodo}
Fix $(a_0,f^{+}_0,f^{-}_0)$ as above with $a_0\leq r^{6}$ and $|f_{0}^{+}|, |f_0^{-}|\leq 2\mu r$. There exists $K=K(\mu)$ sufficiently large, such that for $\lambda$ sufficiently small and for some constant $c=c(\mu,\lambda)>0$, for all large enough $r,$
\begin{equation}
\label{e:smallodo}
\sP_{*,\hat{\eta}}(\A\cap \mathcal{D}(a_0,f^{+}_0,f^{-}_0))\leq e^{-cr}.
\end{equation}
\end{lemma}

\vspace{0.2in}
\begin{center}
\textbf{For the rest of this subsection, fix $(a_0,f^{+}_0,f^{-}_0)$ as in Lemma \ref{l:smallodo}. Recall also that $\hat{\eta}$, the initial configuration of active particles, supported on $K\Z \cap [-r,r]$  has already been fixed.}
\end{center}
\vspace{0.3 in}
Given the stack at $0$, (i.e.\ $\{\zeta_{(0,j)}:j\in \N\}$) using Lemma \ref{crucial} and the fact that $M(0)=a_0$, one can compute the value of $R(0\to)$ (say $r_0$). Since $F^{+}_{0}=f^+_0,$ by definition  \eqref{fluxnot1}, the value of $L(0\leftarrow 1)$ is 
\begin{equation}\label{fluxcond1}
\ell_{1}=r_0-f^+_0.
\end{equation}

Now $\ell_1$, together with the stack at $K$, i.e., $\{\zeta_{1,j}:j\in \N\}$, can be used to calculate $M(1)$ up to some (typically small) error. In general, given $\ell_{x}$, one can work out the value of $M(x)$ up to some error. Let $\ell^{SS}_{x}(\cdot):=L^{SS}_{x-1\leftarrow x}(\cdot) $ where the latter was defined in \eqref{ssd1}. We have the following lemma.
\begin{lemma}\label{monojump}
For any $x\in \Z,$ and every realization of the stack  at $xK,$ (i.e.,\ $\{\zeta_{(x,j)}:j\in \N\}$, (see \eqref{stack2} ) the function $\ell^{SS}_x(\cdot)$ is a non decreasing function with jump size at most $2K$. 
\end{lemma}
\begin{proof}
Notice that  by definition for $j\geq 0$, we have  $\ell^{SS}_x(j+1)\ge \ell^{SS}_x(j).$ This is because by the toppling rule of the single site dynamics in \S~\ref{ssd}, always the active particle with the least label gets toppled. Thus before the particle labelled $(x,j+1)$ gets emitted, all the particles with smaller labels are either asleep in the interval $((x-1)K, (x+1)K)$ (at most $2K-1$) or have hit $\{(x\pm1)K\},$ and hence at this point the number of particles to hit $(x-1)K$ is precisely $\ell^{SS}_x(j)$. Since the emission of the particle $(x,j+1)$ can lead to other sleepy particles to wake up and hit $(x-1)K$, 
the above inequality follows.
Moreover, as the total number of particles at this stage, i.e. $|\eta^{SS}_{x}(j)|+1$ (the additional one is due to the particle $(x,j+1)$) and by the above discussion, $|\eta^{SS}_{x}(j)|\le 2K-1,$ the bound on the jump size of the process $\ell^{SS}_x(\cdot)$ follows. 
\end{proof}
Now by Lemma \ref{crucial}, $\ell_1=\ell^{SS}_1(M(1)).$
However since the function $\ell^{SS}_1(\cdot),$ has jumps, it might turn out that for some realization of the stack $\{\zeta_{(1,j)}:j\in \N\}$, there does not exist any $j$ such that  $\ell_1^{SS}(j)=r_0-f^+_0$ and hence $\mathcal{D}(a_0,f_0^{+},f_0^{-})$ cannot hold. On the other hand, since the function $\ell^{SS}_1(\cdot)$ is not strictly increasing, it might also happen that $\ell^{SS}_1(j)=r_0-f^+_0,$ for more than one value of $j.$
In that case, let 
\begin{equation} \label{inverse}
[M_*(1),M_*(1)+w(1)],
\end{equation}
for some positive integer $M_*(1)$ and $w(1)\geq 0,$ be the interval of all values $y$ such that   $\ell^{SS}_{1}(y)=r_0-f^+_0.$ In this case, $M(1)$ cannot be determined just by 
$\mathcal{D}(a_0,f_0^{+},f_0^{-})$ and the stacks at $0$ ($\{\zeta_{0,\cdot}\}$) and at $K$ ($\{\zeta_{K,\cdot}\}$) as it can be any number in the interval $[M_*(1),M_*(1)+w(1)]$.

We now formally carry out the strategy outlined in the roadmap.

\begin{definition}\label{sleepnot}
Let $\tilde S_{i,i+1}(j)$ be the number of labeled sleepy particles in the configuration $\eta^{SS}_{x}(j)$ (by definition all of them have label $(i,\cdot)$) between $(Ki,K(i+1))$ and similarly $\tilde S_{i,i-1}(j)$ denotes the number of sleepy $(i,\cdot)$ labeled particles between $(K(i-1),K(i)].$  Also following Lemma \ref{crucial} let $$\tilde S_{i,i+1}:=\tilde S_{i,i+1}(M(i)) \mbox{ and } \tilde S_{i,i-1}:=\tilde S_{i,i-1}(M(i)),$$ 
be the $(i,\cdot)$ labeled sleepy particles in $\hat{\eta}_{x,\infty}$.
\end{definition} 

Fix a sequence of non negative integers, 
\begin{equation}\label{treepath}
\underline{y}=\left\{ y(-\frac{2r}{K}+1), \ldots, y(-2),y(-1),y(0)=0,y(1),y(2), \ldots , y(\frac{2r}{K}-1)\right\}.
\end{equation}
\textbf{All the subsequent definitions will implicitly depend on the sequence $\underline y$ and we will not mention it every time. Also for brevity we will choose to suppress this dependence, in the notation. Since $\underline y$ will be fixed, this should not create any scope for confusion.}

Given the triple $(a_0,f^{+}_0,f^{-}_0),$ we now define a filtration $\{\mathscr{F}_i\}_{i\ge 0} $ generated by the variables $\zeta_{(\cdot,\cdot)}$ (also depending on the initial particle configuration $\hat{\eta}$) as follows: For convenience of reading, we shall denote the filtration by the set of random variables that generate it. We start by defining,
\begin{align}\label{fil1}
\sF_0 :=\{\zeta_{(0,1)}, \zeta_{(0,2)},\ldots \zeta_{(0,a_0)}\}. 
\end{align} 

We now recursively define the potential values $f^{+}_i, f^{-}_i$ of the functions $F^+_{i}, F^-_{i}$ respectively etc. These values will be obtained as a deterministic function of the filtration $\sF_i$ (yet to be fully defined). For $i\geq 1$, assume $\ell_1,\ldots , \ell_i$; $a_1,\ldots , a_i$ and $f_0^+,\ldots, f_{i-1}^+$ have already been defined (note that $\ell_1$ and $f_0^+$ have already been defined). Let
\begin{equation}
\label{ind1}
\ell_{i+1}:=R^{SS}_{i\to}(a_{i})-f^+_i
\end{equation}

We now define the filtration $\sF_i$ recursively for $\underline{y}$ such that $y(0),y(1),\ldots ,y(i)$ is \textbf{realizable} (the definition of realizable is also recursive and defined below) . Assume $\sF_0,\ldots \sF_i$,  have been defined. Note that $\ell_{1}$ will be a deterministic function of $\sF_0,$ (given the other fixed data) and similarly $\ell_{i+1}$ will be a deterministic function of $\sF_i,$ (see below)  Analogous to \eqref{inverse}, define $M_*(i+1)$ and $w(i+1)$ by 
\begin{equation}\label{pos45}
M_*(i+1)=\min\{j\geq 0: \ell_{i+1}^{SS}(j)=\ell_{i+1}\};~~ M_*(i+1)+w(i+1)=\max\{j\geq 0: \ell_{i+1}^{SS}(j)=\ell_{i+1}\}
\end{equation}
if such numbers exist. Call the sequence $\{y(0),y(1),\ldots y(i+1)\}$ \textbf{realizable} if (a) either $i=0$ or $\{y(0),y(1),\ldots y(i)\}$ is \textbf{realizable} and (b) $M_{*}(i+1)$ and $w(i+1)$ as above exist and $y(i+1)\leq w(i+1)$. In such a situation define 
\begin{equation}
\label{ind3}
a_{i+1}=M_*(i+1)+y(i+1).
\end{equation}
For a realizable sequence $\{y(0),y(1),\ldots y(i+1)\}$, define  $$\sF_{i+1}:=\sF_i \bigcup \{\zeta_{(i+1,1)}, \zeta_{(i+1,2)},\ldots \zeta_{(i+1,a_{i+1})}\}.$$
Now, by definition $R^{SS}_{i+1 \to }(a_{i+1}),\tilde S_{i+1,i+2}(a_{i+1})$ are $\sF_{i+1}$ measurable. 
Thus as mentioned above the quantity $\ell_{i+1}$ is $\sF_{i}$ measurable.

We now define the potential value of the flux $F^+_{i+1}$ to be 
\begin{equation} 
\label{ind2}
f^+_{i+1}:=\biggl|\hat{\eta}_{(iK+1,(i+1)K]}\biggr| + f^+_i-\tilde S_{i,i+1}(a_i)-\tilde S_{i+1,i}(a_{i+1});
\end{equation}
where $|\hat{\eta}_{B}|$ denote the number of particles of $\hat{\eta}$ contained in $B\subseteq \Z$ and the notation $S_{i,i+1}(\cdot)$ for sleepy particles was defined in Definition \ref{sleepnot}. 
That the above definition is natural, is immediate by observing that 
$f^+_{i}-f^+_{i+1}$ is the net number of particles coming in to the interval $$(Ki+\frac{1}{2},K(i+1)+\frac{1}{2}).$$ The latter quantity is the same as the final number of sleepy particles ($\tilde S_{i,i+1}(a_i)+\tilde S_{i+1,i}(a_{i+1})$) minus the initial number of particles,  $\left|\hat{\eta}_{\bigl(iK+1,(i+1)K\bigr]}\right|,$ (note that the interval is semi-closed as we take in to account the site $(i+1)K$).
That  $f^+_{i+1}$ is $\sF_{i+1}$ measurable follows from definition. 

The above discussion, specifically \eqref{ind1}, \eqref{ind2} and \eqref{ind3} describe the construction of the filtration  $\sF_i$ for all $i\ge 0$ as long as the realizable conditions in \eqref{ind3} are met. 
Now for $i<0$, we construct in the exactly similar way a filtration $\sF_{i}$, and a sequence of functions $M_*(i)$, $w(i)$, $a_i$, $r_{i}$ and $f^{-}_{i}$, recursively going from right to left and starting with $(a_0,f^{-}_{0})$. The sequence $\{y(i),\ldots ,y(0)\}$ is called realizable if $y(i)\leq w(i)$. 

Finally, we call $\underline{y}$ realizable if $$\{ y(-\frac{2r}{K}+1), \ldots, y(-2),y(-1),y(0)\} \mbox{ and } \{y(0),y(1),y(2), \ldots , y(\frac{2r}{K}-1)\}$$ are both realizable.

We now have the following lemma whose proof follows by definition:
\begin{lemma}
\label{yexist}
On $\mathcal{D}(a_0,f^+_0,f^{-}_0)\cap \A$, there exists a (unique) $\underline{y}$ as in \eqref{treepath} such that $\underline{y}$ is realizable and $M(i)=M_*(i)+y(i)$ for each $i$. Further for such an $\underline{y}$, we have 
$L(j-1\leftarrow j)=\ell_j$ and $F_j^{+}=f^{+}_{j}$ for all $j\geq 0$ where 
$\ell_j$ and $f_j^{+}$ are defined as in \eqref{ind1} and \eqref{ind2} respectively.
\end{lemma} 
\begin{proof} By definition on $\mathcal{D}(a_0,f^+_0,f^{-}_0)\cap \A,$ recursively according to the definitions in \eqref{pos45}, for any $i,$ the renormalized odometer values at $M(i)$ must belong to the interval $[M_*(i+1),M_*(i+1)+w(i+1)],$ (note that this in particular implies that the interval $[M_*(i+1),M_*(i+1)+w(i+1)]$ must exist, as discussed right after \eqref{inverse}). Thus the lemma follows by inductively defining $$y_i=M(i)-M_*(i).$$ The statements about $F_j^{+}$ and $L(j-1\leftarrow j)$ follow directly from definitions. 
\end{proof}
It might be instructive to think of the sequence $\underline y$ as a path in a tree of possibilities of the odometer values at various points, out of which the process picks only one. 
\begin{figure}[hbt]
\includegraphics[scale=.7]{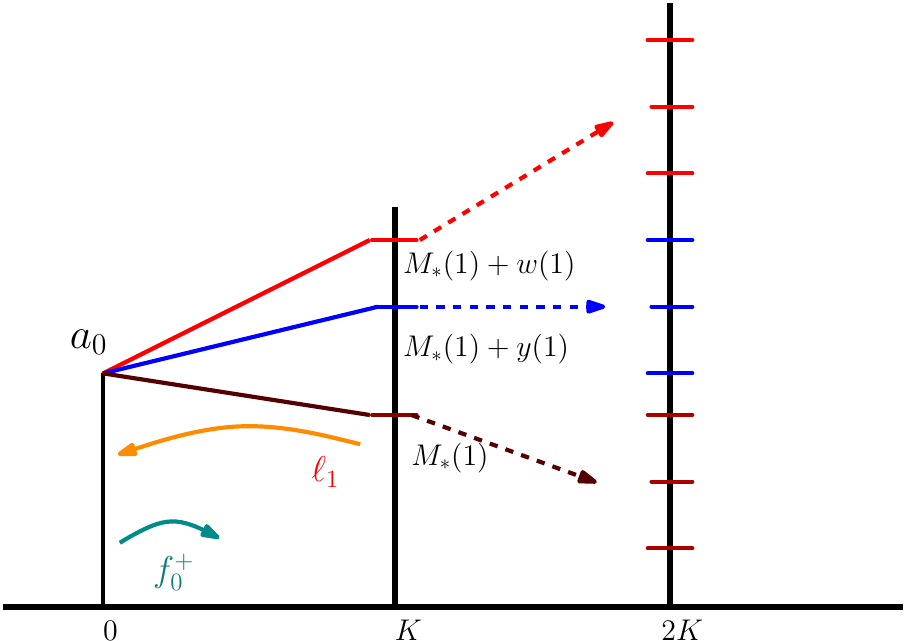}
\caption{The paths of various colors denote the different realizable sequences.}
\label{tree1}
\end{figure}

\begin{definition}\label{realizable}
Let $\mathcal{R}(\underline{y})$ be the event that $\underline{y}$ is realizable and $M(i)=M_*(i)+y_i$ for all $i$. Also for $j>0$ (resp.\ $j<0$), let $\mathcal{R}(\underline{y};j)$ denote the event that $\{y(0),\ldots y(j)\}$ is realizable (resp.\ $\{y(j),\ldots ,y(0)\}$ is realizable) and $M(i)=M_*(i)+y_{i}$ $\, \forall\, i\in \{0,1,\ldots , j\}$ (resp.\ $\, \forall \,i\in \{j,j+1,\ldots ,0\}$). 
\end{definition}
Recalling the event $\cD=\mathcal{D}(a_0,f^+_0,f^{-}_0)$
we now state for any $\underline{y}$ and the triple $(a_0,f^+_0,f^{-}_0)$ as in Lemma \ref{l:smallodo}, a quantitative upper bound on the probability of the event $\mathcal{R}(\underline{y})\cap \mathcal{D}\cap \A$. Also recall that the initial configuration $\hat \eta,$ satisfies $|\hat \eta|\ge \frac{\mu r}{2}.$
\begin{lemma}
\label{quantbnd1}
For any $K$ sufficiently large depending on $\mu$, there exists $\lambda_K$ sufficiently small, such that for $\lambda\le \lambda_K$,  and  $\underline{y}$ as in \eqref{treepath}, 
$$\sP_{*,\hat{\eta}}( \mathcal{R}(\underline{y})\cap \mathcal{D}\cap \A)\le  e^{-\mu r/4} e^{2K} \left((.52)^ {\sum_{i=-2r/K+1}^{1}y_i}+(.52)^ {\sum_{i=1}^{2r/K-1}y_i}\right).$$
\end{lemma}

Fix $\underline{y}$ as in \eqref{treepath}. To prove Lemma \ref{quantbnd1}, under the same hypothesis as the latter, we prove the following bound on the exponential moment of the number of sleepy particles:
\begin{lemma}\label{expmomentlem2}
\begin{equation}\label{expmoment}
\E \left(e^{\sum_{i=0}^{2r/K-1}[\tilde S_{i,i+1}+\tilde S_{i,i-1}]}\mathbf{1}[\mathcal{R}(\underline{y})\cap \mathcal{D}]\right)\le e^{2K}(.52)^{{\sum_{i=0}^{2r/K-1}y_i}}.
\end{equation}
A similar bound for the negative axis holds.
\end{lemma}
Before proceeding further we quickly finish the proof of Lemma \ref{quantbnd1} using the above.
\begin{proof}[Proof of Lemma \ref{quantbnd1}]
Observe that on the event $\mathcal{D}\cap \A,$ trivially either 
\begin{align}\label{lowerbnd123}
e^{\sum_{i=0}^{2r/K-1}[\tilde S_{i,i+1}+\tilde S_{i,i-1}]} \ge e^{\mu r/4}\text{ or }\,\,
e^{\sum_{i=-2r/k+1}^{0}[\tilde S_{i,i+1}+\tilde S_{i,i-1}]} \ge e^{\mu r/4},
\end{align}
(follows from the observation recorded as  Lemma \ref{totalsleep}).
Thus using \eqref{expmoment} and \eqref{lowerbnd123} we get 
\begin{align}\label{appl567}
e^{\frac{\mu r}{4}}\sP_*(\mathcal{R}(\underline{y})\cap \mathcal{D}\cap \A) &\le \E \left(e^{\sum_{i=0}^{2r/K-1}[\tilde S_{i,i+1}+\tilde S_{i,i-1}]}\mathbf{1}\bigl[\mathcal{R}(\underline{y})\cap \mathcal{D}\bigr]\right)\\
&+ \E \left(e^{\sum_{i=-2r/K+1}^{0}[\tilde S_{i,i+1}+\tilde S_{i,i-1}]}\mathbf{1}\bigl[\mathcal{R}(\underline{y})\cap \mathcal{D}\bigr]\right),\\
\nonumber
&\le e^{2K}\left((.52)^ {\sum_{i=-2r/K+1}^{1}y_i}+(.52)^ {\sum_{i=1}^{2r/K-1}y_i}\right).
\end{align}
\end{proof}

We now proceed toward proving \eqref{expmoment}. We shall prove the following lemma.
\begin{lemma}\label{induction675}
For all $i \in \{1,\ldots ,2r/K-1\}$: 
\begin{equation}\label{indbound}
\E\left[e^{[\tilde S_{i,i+1}+\tilde S_{i,i-1}]}\mathbf{1}\bigl[\mathcal{R}(\underline{y}; i)\cap \mathcal{D} \bigr] \mid \sF_{i-1}\right]\le (.52)^{y_i}\mathbf{1}\bigl[\mathcal{R}(\underline{y}; i-1)],
\end{equation}
and similarly for $i\in \{-2r/K+1, \ldots , -1\}$,
\begin{equation}\label{indboundneg}
\E\left[e^{[\tilde S_{i, i+1}+\tilde S_{i, i-1}]}\mathbf{1}\bigl[\mathcal{R}(\underline{y}; i)\cap \mathcal{D} \bigr] \mid \sF_{i+1}\right]\le (.52)^{y_i}\mathbf{1}\bigl[\mathcal{R}(\underline{y}; i+1)].
\end{equation}
\end{lemma}

Assuming the above, we quickly finish the proof of \eqref{expmoment}.
\begin{proof}[Proof of Lemma \ref{expmomentlem2}]
Since for $i>0,$ $\cR(\underline y;i)$ and $\tilde S_{i,i+1}+\tilde S_{i,i-1}$ are $\sF_{i}$ measurable and a similar statement holds for $i<0$. Using the  above bounds,
and induction on $i$, \eqref{expmoment} follows. 
Fixing $i=2r/K-1,$ formally notice that, 
\begin{align*}
&\E \left(e^{\sum_{j=0}^{2r/K-1}[\tilde S_{j,j+1}+\tilde S_{j,j-1}]}\mathbf{1}[\mathcal{R}(\underline{y})\cap \mathcal{D}]\right)\\
&= \E \left(e^{\sum_{j=0}^{2r/K-2}[\tilde S_{j,j+1}+\tilde S_{j,j-1}]}\mathbf{1}[\mathcal{R}(\underline{y}; 2r/K-2)\cap \mathcal{D}]\E\biggl[e^{[\tilde S_{i,i+1}+\tilde S_{i,i-1}]}\mathbf{1}[\mathcal{R}(\underline{y}; i)\cap \mathcal{D}]  \mid \sF_{i-1}\biggr]\right)
\\
&\overset{\eqref{indbound}}{\le}\E \left(e^{\sum_{j=0}^{2r/K-2}[\tilde S_{j,j+1}+\tilde S_{j,j-1}]}\mathbf{1}[\mathcal{R}(\underline{y}; 2r/K-2)\cap \mathcal{D}]\right)(.52)^{y_i}\\
&\le e^{2K}(.52)^{{\sum_{i=0}^{2r/K-1}y_i}}.
\end{align*}
The first equality follows by noticing that the first term in the product i.e., $$e^{\sum_{j=0}^{2r/K-2}[\tilde S_{j,j+1}+\tilde S_{j,j-1}]}\mathbf{1}[\mathcal{R}(\underline{y}; 2r/K-2)\cap \mathcal{D}],$$ is $\sF_{i-1}$ measurable and the last step follows by induction.
We naively bound $\tilde{S}_{0,1}+\tilde{S}_{0,-1},$ by ${2K}$  (since at most $2K$  particles with first label $0$ can fall asleep in the interval $(-K,K)$ ).  This accounts for the extra $e^{2K}$ term in \eqref{expmoment}.
 \end{proof}
 Note that from the bound in \eqref{expmoment} by summing over $\underline y$ and all the possible values of $M(0)\le r^6,$ and $|f^+_0|,|f^-_0|\le 2\mu r$ as in Lemma \ref{l:smallodo}, we get  for all large enough $r,$
\begin{align}\label{expboun768}
\E \left((e^{\sum_{i=-2r/K+1}^{0}[\tilde S_{i,i+1}+\tilde S_{i,i-1}]}+e^{\sum_{i=1}^{2r/K-1}[\tilde S_{i,i+1}+\tilde S_{i,i-1}]})\mathbf{1}[M(0)\le r^6]\right)\le e^{2K}e^{cr}
\end{align} 
 where $c=c(K)$ can be made arbitrarily small by choosing $K$ large enough. \footnote{It is worth noting that if  $S_1, S_2$ denote the total number of sleepy particles in $[-2r ,0]$ and $[0,2r]$ respectively after stabilization, by definition and the Abelian Property, it follows that 
$S_1 \le \sum_{i=-2r/K+1}^{0}[\tilde S_{i,i+1}+\tilde S_{i,i-1}],$ and a similar bound is true for $S_2.$ Thus the above along with the fact that $M(0)\le u_{[-2r,2r]}(0),$ (see Lemma \ref{l:largeodo}) implies that, for any $c>0,$ by choosing 
$\lambda$ small enough, for all large enough $r,$
 \begin{align}\label{expboun7698}
\E \left((e^{S_1}+e^{S_2})\mathbf{1}[u_{[-2r,2r]}(0)\le r^6]\right)\le e^{cr}. 
\end{align} }

  We now proceed to proving  \eqref{indbound}. The proof of \eqref{indboundneg} will follow similarly. 
However we need some notational preparation: recall that the random walk paths $\zeta_{(\cdot,\cdot)}$ in \eqref{stack2} are lazy (with parameter $\frac{\lambda}{1+\lambda}$) stopped on hitting the nearest renormalized lattice point (multiple of $K$) to the left or to the right. By symmetry each random walk path has probability $1/2$ of hitting either neighbour. Call a random walk path $\zeta_{(x,\cdot)}$ a \textbf{left instruction} (resp.\ \textbf{right instruction}) if it is killed at hitting $(x-1)K$ (resp.\ $(x+1)K$) before taking any lazy step. Clearly for any $\chi>0,$ and $\lambda=\lambda(K,\chi)$ sufficiently small, 
\begin{align}
\label{probbound342}
\sP_{*}(\zeta_{(x,\cdot)}~\text{is a \bf{left instruction}})&\geq \frac{1}{2}-\chi \\
\nonumber
\sP_{*}(\zeta_{(x,\cdot)}~\text{is a \bf{right instruction}})&\geq \frac{1}{2}-\chi.
\end{align}
Fix any $\e>0$. Let $R$ be large enough such that 
\begin{equation}\label{def978}
\sP_{*}(\zeta_{(x,j)}~\text{is a {\bf{right instruction}} for some } 1\le j \le R )\geq 1-\e/8.
\end{equation}
 For every $x\in \Z$, given the value $\ell_{x}$ (see \eqref{ind1} and note that it is measurable with respect to $\sF_{x-1}$), we define the following stopping times measurable with respect to the filtration
\begin{equation}\label{localfil}
\sG_{x}(j)=\{\zeta^{(j)}_{(x,1)},\ldots \zeta^{(j)}_{(x,j)}\}
\end{equation}
where $\zeta^{(j)}_{(x,\cdot)}$ denotes the part of the random walk trajectory $\zeta_{(x,\cdot)}$ used in running the single site dynamics starting with $j$ active particles at $xK$. 

Recall $\ell^{SS}_{x}(\cdot)$ from Lemma \ref{monojump} and for $R$ as above, define
\begin{align}\label{stop1}
\tau_1 :=\tau_1(x) := \inf\{j: \ell^{SS}_{x}(j) \ge \ell_x -3KR \}.
\end{align}
Let  $\tau_2=\tau_1+R$. Also let
$$ \tau_3=\inf\{j: \ell^{SS}_{x}(j)-\ell^{SS}_{x}(\tau_1) \ge 3KR \}$$ i.e., $\tau_3$ is the first time after $\tau_1$ such that there are at least  $3KR$ many $\textbf{left instruction}$s among the random walk paths $\zeta_{(x,\tau_1+1)}\ldots \zeta_{(x,\tau_3)}$. Clearly $\tau_3\ge \tau_1+3KR$ and hence $\tau_3 > \tau_2.$ 
Also notice that using Lemma \ref{monojump}, on the event that $M_*(x)$ exists, $$M_*(x) \in [\tau_2,\tau_3].$$ 

We now bound the maximum number of sleepy particles with first label $x$ ever to be on the interval $\bigl((x-1)K,(x+1)K\bigr)$ in the time interval, $[\tau_2,\tau_3],$ i.e.,\ $$\max_{\tau_2 \le  j\le \tau_3} \tilde S_{x,x+1}(j)+\tilde S_{x,x-1}(j).$$ 
 
To do this,  we specify our choice for $\lambda.$  
\begin{equation}\label{def979}
\sP_{*}(\zeta_{(x,j)} \text{ does not have a lazy step for } 1\le j \le 2K \text{ and }  \tau_1 \le j \le \tau_3  )\geq 1-\e/8.
\end{equation}
Note that by choosing $\lambda$ small enough dependent on $\e,K$, both \eqref{def978} and \eqref{def979} is satisfied. We might need to choose $\lambda$ to be even smaller, depending on some conditions that will appear in the subsequent proofs. 
We now have the following lemma.
\begin{lemma}\label{probbound675}  With the choice of $\lambda$ satisfying \eqref{def979}, conditional on the randomness exposed at sites less than $xK$ and at $x$ up to $\tau_1$ i.e., $\sF_{x-1} \cup \sG_x(\tau_1),$ the following uniform bound on the number of sleepy particles   $S_{x,x+1}(j)+\tilde S_{x,x-1}(j)$ holds:  $$\sP_{*}(\max_{\tau_2 \le  j\le \tau_3} \tilde S_{x,x+1}(j)+\tilde S_{x,x-1}(j)=0\mid \sF_{x-1} \cup \sG_x(\tau_1))\ge 1-\e.$$
\end{lemma}
 \begin{proof}
Conditioning on the filtration $\sF_{x-1}\cup \sG_{x}(\tau_1)$,  consider the configuration, $\eta^{SS}_{x}(\tau_1)$ (recall from the beginning of \S~\ref{ssd}, that $\eta^{SS}_{x}(\tau_1)$ denotes the configuration of sleepy particles with labels of type $(x,\cdot)$ after $\tau_1$ rounds of the single site dynamics on the interval $((x-1)K, (x+1)K)$). Clearly $|\eta^{SS}_{x}(\tau_1)|\leq 2K$.

By choice of $R,$ with probability at least $1-\e/4$ there has been at least a $\textbf{left instruction}$ and a $\textbf{right instruction}$ in the interval $(\tau_1,\tau_2]$. Let 
$$s_1=\min\{k>\tau_1: \zeta_{(x,k)}~\text{is a \bf{left instruction}}\},~s_2=\min\{k>\tau_1: \zeta_{(x,k)}~\text{is a \bf{right instruction}}\}.$$ 
Moreover by \eqref{def979} with failure probability at most $\e/8$, no random walk $\zeta_{(x,j)}$ for $j\in [\tau_1,\tau_3]$ will have any lazy step and hence no new particle falls asleep. 

Now considering the single site dynamics at $x,$ from time $\tau_1$ to $\tau_2=\tau_1+R,$ at time $s=\max(s_1,s_2)$, all the sleepy particles in $\eta^{SS}_{x}(\tau_1)$ (at most $2K$) have woken up on the interval $((x-1)K,x]$ and $(x,(x+1)K)$ since $\zeta_{(x,s_1)}$ or $\zeta_{(x,s_2)}$  hits  them.
Now we claim that none of these woken up particles will fall asleep again before hitting  $\{x\pm 1\}K$ with probability at least $1-\e/8.$ Indeed, observe that these particles follow lazy random walk trajectories starting from vertices  in $((x-1)K, (x+1)K)$ that are measurable with respect to $\sG_x(\tau_1)$, but the sequence of lazy random walk steps taken by the particles are independent of $\sG_x(\tau_1)$ (by definition, these steps were not revealed in $\sG_x(\tau_1)$). Observe also that the probability of there being no lazy step before a lazy random walk started at $y$ in the interval $((x-1)K, (x+1)K)$ exits the same is minimized when $y=xK$. The claim therefore follows from \eqref{def979}. 

Putting things together, with probability at least $1-\e$, we have $\tilde S_{x,x-1}(j)+\tilde S_{x,x+1}(j)=0$ for all $j\in [s,\tau_3]$ and this completes the proof. 
\end{proof}
For notational convenience, let us denote the complement of the event in Lemma \ref{probbound675} by $\mathcal{B}_1,$ which occurs with probability at most $\e.$
Let us now return to the proof of \eqref{indbound}.
\subsubsection*{Proof of Lemma \ref{induction675}.}
Fix $i>0$. We want to prove 
$$\E\left[e^{[\tilde S_{i,i+1}+\tilde S_{i,i-1}]}\mathbf{1}\bigl[\mathcal{R}(\underline{y};i)\cap \mathcal{D} \bigr] \middle| \sF_{i-1}\right]\le (.52)^{y_i}.
$$
Recall from previous discussion (also see using \eqref{ind1}), that $\ell_{i}$ is $\sF_{i-1}$ measurable. 
We will also need the filtration $\sG_{i}(j)$ from \eqref{localfil} and the stopping time $\tau_1=\tau_1(i)$ from \eqref{stop1}.  We compute 
$$\E\left[e^{[\tilde S_{i,i+1}+\tilde S_{i,i-1}]}\mathbf{1}\bigl[\mathcal{R}(\underline{y};i)\cap \mathcal{D} \bigr] \middle| \sF_{i-1} \cup \sG_{i}(\tau_1) \right].$$

Now note that given $\cR(\underline{y};i-1),  \sF_{i-1} \cup \sG_{i}(\tau_1)$, for the event $\cR(\underline{y};i-1)$  to occur, two things need to hold: 
\begin{enumerate}
\item  $M_*(i)$ exists. 
\item  $\ell^{SS}_{i}(M_*(i)+y_i)=\ell^{SS}_i(M_{*}(i))$.
\end{enumerate}
Let us assume that the first event holds. In that case, 
conditional on $ \sF_{i-1} \cup \sG_{i}(M_*(i)),$ clearly, by definition, for this to happen, $\zeta_{(i,k)}$ cannot be a $\textbf{left instruction}$ for any $k\in [M_*(i)+1, M_*(i)+y_i]$. Let $\mathcal{B}_2$ denote this event. Now by \eqref{probbound342}, for $\lambda$ small enough, the probability of a $\textbf{left instruction}$ is at least $.49.$ Thus,
\begin{equation}
\label{e:b3bound}
\sP_{*}[\mathcal{B}_2\mid \sF_{i-1} \cup \sG_{i}(\tau_1)] \leq (0.51)^{y_i},
\end{equation}
  
Now we fix a constant $C$ to be specified later (whose purpose will be clear soon).
Let $\mathcal{B}_3$ denote the event that at least one of the random walks $\zeta_{(i,M_*(i)+1)},\ldots,\zeta_{(i,M_*(i)+C)} $ takes a lazy step (see \eqref{steps1}). By taking $\lambda=\lambda(C)$ small enough,
\begin{equation}
\label{e:b4bound}
\P[\mathcal{B}_3\mid \sF_{i-1} \cup \sG_{i}(\tau_1)] \leq \e.
\end{equation}
The above choices allow us to conclude that 
$$
\E\left[e^{[\tilde S_{i,i+1}+\tilde S_{i,i-1}]}\mathbf{1}\bigl[\mathcal{R}(\underline{y};i)\cap \mathcal{D} \bigr]  \middle|  \sF_{i-1} \cup \sG_{i}(\tau_1) \right ]$$
\begin{align}\label{bound23}
\leq \left \{ 
\begin{array}{cc}
e^{2K}(0.51)^{y_i} & \text {if } y_i\ge C\\
3\e e^{2K}+(0.51)^{y_i} & \text {if } y_i<  C.
\end{array}
\right .
\end{align}
The first case in the above computation is straightforward: we use the fact $\tilde S_{i,i+1}+\tilde S_{i,i-1}$ is at most $2K$,  and
$\mathcal{R}(\underline{y};i)\cap \mathcal{D}  \subset \mathcal{B}_2$ (clear from the discussion preceding the definition of $\mathcal{B}_2$),
 and \eqref{e:b3bound}.
 
For the second case, first  recall the set $\cB_1,$ (complement of the set defined in Lemma \ref{probbound675}). 
Now notice that for $y_i<C$, by union bound, and the set inclusion just mentioned above,
\begin{align*}
\E\left[e^{[\tilde S_{i,i+1}+\tilde S_{i,i-1}]}\mathbf{1}\bigl[\mathcal{R}(\underline{y};i)\cap \mathcal{D}\bigr] \mid  \sF_{i-1} \cup \sG_{i}(\tau_1) \right]\le & \P(\mathcal{B}_1 \cup \mathcal{B}_3 \mid \sF_{i-1} \cup \sG_{i}(\tau_1) )e^{2K}+\\
& \P(\mathcal{B}_1^{c}\cap \mathcal{B}_3^{c} \cap \mathcal{B}_2\mid \sF_{i-1} \cup \sG_{i}(\tau_1) ).
\end{align*}

In the first term on the RHS, we use the crude bound $\tilde S_{i,i+1}+\tilde S_{i,i-1} \le 2K.$  In the second term, by definition, on the event $\cB^c_1,$ we have $\tilde S_{i,i+1}+\tilde S_{i,i-1}=0$
The second conclusion in \eqref{bound23} now follows from Lemma \ref{probbound675},  \eqref{e:b3bound} and \eqref{e:b4bound}. Now we fix $C>0$ (depending on $K$) such that the bound in \eqref{bound23} in the case $y_i\ge C$ is in fact less than $(0.52)^{y_i}$. To have the same bound of $(0.52)^{y_i}$ in the case $y_i<C,$
we choose $\e>0$  small enough (depending on $C$). Thus averaging \eqref{bound23} over $\sG_{i}(\tau_1)$  we get that, 
$$\E\left[e^{[\tilde S_{i,i+1}+\tilde S_{i,i-1}]}\mathbf{1}\bigl[\mathcal{R}(\underline{y};i)\cap \mathcal{D} \bigr] \middle | \sF_{i-1}\right]\le (0.52)^{y_i}.$$ This completes the proof of \eqref{indbound} and \eqref{indboundneg} can be established similarly. \qed

Using Lemma \ref{quantbnd1}, we complete the proof of Lemma \ref{l:smallodo}.
\begin{proof}[Proof of Lemma \ref{l:smallodo}]
Note that by Lemma \ref{yexist}, and union bound, it follows that that 
\begin{align}\label{bnd1part}
\sP_{*,\hat{\eta}}(\mathcal{D}(a_0,f_0^{+},f_0^{-})\cap \A)&\le \sum_{\underline y}\sP_{*,\hat{\eta}}(\mathcal{D}(a_0,f_0^{+},f_0^{-})\cap \cR(\underline y)\cap \A),\\
\nonumber
& \le 
 \sum_{\underline y}  e^{-\mu r/4} e^{2K} \left((.52)^ {\sum_{i=-2r/K+1}^{1}y_i}+(.52)^ {\sum_{i=1}^{2r/K-1}y_i}\right)\\
\label{finalcal123}
&\le 2e^{2K} e^{-\mu r/4}\left(\frac{1}{0.48}\right)^{2r/K}\leq e^{-cr}. 
\end{align}
for some $c>0,$ by choosing $K=K(\mu)$ to be large enough, where the last inequality holds for large enough $r$. This establishes \eqref{e:smallodo} and completes the proof.
\end{proof}
\begin{rem}\label{relation4532} Although highly suboptimal, we point out that the for small $\mu$, the bound on $\lambda$ obtained from the above argument is exponentially small in $\frac{1}{\mu}$ i.e., any $\lambda\le \exp(-\frac{C}{\mu})$ for some universal positive constant $C$ works. To see this,  first note that  \eqref{finalcal123} implies that it suffices to take $K=\frac{D}{\mu}$ for some universal constant $D$.  By the discussion following \eqref{bound23},  it is clear that  $C$ in \eqref{bound23} can be taken to be a universal constant times $K$ and hence $\e$ needs to be exponentially small in $K.$  Fixing such an $\e,$ this implies that $R$ in \eqref{def978} can be taken to be a constant multiple of $K$. Now we determine the choice of $\lambda.$ By taking $\chi$ to be $.1$, we see that $\tau_3=O(K^2),$ with probability at least $1-\frac{\e}{100}.$ Thus $\lambda$ has to be small enough to simultaneously satisfy \eqref{probbound342} and \eqref{def979}. It is easy to see that satisfying the latter for the above choices of $K,  R$ and $\e,$ is enough. 
However to satisfy the latter, one needs 
$\lambda$ small enough to ensure that $O(K^2)$ many  random walk paths starting from the origin, do not fall asleep before hitting $\pm K$ with probability at least $1-\frac{\e}{100}.$ Total number of steps taken by all such random walks is at most $K^6$ with probability at least $1-e^{-\Theta(K^2)},$ and hence taking $\lambda=\frac{\e}{200K^6}$ is sufficient. Since $K^6$ is only a polynomial in $\frac{1}{\mu},$ taking $\lambda$ to be $\e^2$ suffices for all small $\mu.$
\end{rem}
Next we analyze the remaining case i.e.,\ when $M(0)$ is large.
 
\subsection{$M(0)$ is large.}
Recall the renormalized odometer $M(\cdot)$ from \eqref{e:reodo}. To complete the proof of Lemma \ref{key1}, we still need to deal with the case when $M(0)\geq r^{6}$. We have the following lemma.

\begin{lemma}
\label{l:largeodo}
In the set up of Lemma \ref{key1}, we have 
\begin{equation}
\label{e:largeodo}
\sP_{*,\hat{\eta}}(M(0)\geq r^{6}, \A)\leq e^{-cr}
\end{equation}
for some constant $c>0$.
\end{lemma}

\begin{proof}
The proof of this follows easily from random walk estimates.  We will in fact show that the odometer at the origin, of the actual ARW dynamics stabilizing $\eta$ in $[-2r,2r]$, i.e., $u_{[-2r,2r]}(0)$ is smaller than $r^{6}$ at the origin, with probability at least $1-e^{-cr}.$
Now by Lemma \ref{o:ver1comp}, $u_{[-2r,2r]}(0),$ is an upper bound on $M(0)$ and hence we will be done. 
To prove $u_{[-2r,2r]}(0)\le r^6,$ note that by definition each particle in ARW dynamics, follows the trajectory of 
a lazy symmetric random walk with laziness probability $\frac{\lambda}{1+\lambda}$ till it either falls asleep eventually inside $[-2r,2r]$ or exits the interval. 
Now $u_{[-2r,2r]}(0)$ is the total number of times, a particle hits the origin.  Standard random walk facts now imply that the probability that,
started arbitrarily inside $[-2r,2r],$  a lazy random walk does not exit the interval $[-2r,2r]$ in $r^{3}$ steps, is at most  $e^{-cr},$ where $c>0$ depends on the laziness parameter $\frac{\lambda}{1+\lambda}$. Now  since there are at most $2\mu r$ particles, whose trajectories agree with independent random walk paths, by a simple union bound, $\P(u(0)\le 2 \mu r^4)\ge 1-e^{-cr}.$ 
\end{proof}

Finally we put everything together to establish Lemma \ref{key1}.

\subsection{Proof of Lemma \ref{key1}}
By the assumption on $\hat{\eta}$, notice that $|F_0^{+}|, |F_0^{-}|\leq 2\mu r$. Hence taking a union bound over all triples $(a_0,f_0^+,f_0^{-})$ satisfying the conditions in the statement of Lemma \ref{l:smallodo}, it follows from Lemma \ref{l:smallodo} that 
$$\P(\A, M(0)\leq r^{6})\leq e^{-cr/2}.$$
Lemma \ref{key1} now follows from the above bound and Lemma \ref{l:largeodo}.
\qed

\section{Concluding Remarks and Open Questions}
\label{s:question}
In this paper, we have shown that ARW on $\Z$ remains in active phase even starting with arbitrarily low density of particles provided the sleep rate is sufficiently small. In particular this implies that the critical density $\mu_c(\lambda) <1$ for small enough sleep rate $\lambda$. However our  understanding of the process is still far from complete. Investigating  the following seem the  natural next step. 
\begin{enumerate}
\item \textbf{What happens when $\lambda$ is large?} It is believed that for any $\lambda$ finite, $\mu_c{(\lambda)}<1$. However the only known results in this direction is that $\mu_c{(\lambda)}\leq 1$ for all $\lambda$ (\cite{GA10,RS12,Shellef10}) and $\mu_c{(\lambda)}=1$ for $\lambda=\infty$ \cite{CRS14}. As pointed out in Remark \ref{relation4532}, our arguments, yield an exponential relation between $\mu_c$ and $\lambda$. It is natural to conjecture that $\mu_c$ and $\lambda$ should be related quadratically owing to diffusive scaling of random walk. 

\item \textbf{ARW on $\Z^2$:} Does Theorem \ref{main} hold for ARW on $\Z^2$? It is believed that the answer to this question is affirmative. 
In \cite{stauffer} the authors prove Theorem \ref{main}, in the case of all vertex transitive transient graphs which includes $\Z^{d}$ for all $d\ge 3.$

\item \textbf{Critical density for SSM:} For the Stochastic Sandpile Model, is the critical density $\mu_c$ strictly less than one? As mentioned before, numerical evidence suggests an affirmative answer to this question, while the best known rigorous bound in \cite{RS12} gives $\mu_c\leq 1$. A strict inequality here would be a substantial progress because of the same reasons as in Remark \ref{r:more}.    
\end{enumerate}

\section*{Acknowledgements}
The authors thank Vladas Sidoravicius for introducing this problem to them and for many illuminating discussions and Lorenzo Taggi for telling them about his results on the biased activated random walks. RB thanks Allan Sly for useful discussions. The authors also thank two anonymous referees and Leonardo Rolla for many useful comments and suggestions that helped improve the paper. This work was completed while RB and SG were graduate students at the  Department of Statistics at University of California, Berkeley  and the Department of Mathematics at University of Washington, Seattle, respectively. RB was supported by a UC Berkeley Graduate Fellowship.  SG and CH acknowledge support from NSF grant DMS-1308645 and NSA grant H98230-13-1-0827. This work was initiated while RB and SG were interns at Microsoft Research, Redmond. They thank the Theory group for its hospitality.

\bibliography{arw}
\bibliographystyle{plain}
\end{document}